\documentclass{amsart}
\usepackage{a4wide,amssymb}
\usepackage[all]{xy}
\usepackage{color}%,showlabels}%,refcheck}

\newcommand{\w}{\omega}

\newcommand{\Ra}{\Rightarrow}
\newcommand{\F}{\mathcal F}
\newcommand{\U}{\mathcal U}
\newcommand{\V}{\mathcal V}
\newcommand{\W}{\mathcal W}

\newcommand{\Tau}{\mathsf T}
\newcommand{\pr}{\mathrm{pr}}
\newcommand{\cl}{\mathrm{cl}}
\newcommand{\Wallman}{\mathsf W}
\newcommand{\Wk}{\mathsf W_{\!\bar \kappa}}
\newcommand{\Wkstar}{\mathsf W^\star_{\!\beta \kappa}}
\newcommand{\WM}{\Wallman_{\!M}}
\newcommand{\WMstar}{\Wallman^\sharp_{\!M}}

\newcommand{\HM}{\mathsf H_{\!\!\;M}}
\newcommand{\UM}{\mathsf U_{\!\!\;M}}

\newcommand{\Lawman}{\beta^\circ\!}
\newcommand{\OF}{\varphi^\circ}

\newcommand{\Hk}{\mathsf H_\kappa}

\newtheorem{theorem}{Theorem}[section]
\newtheorem{problem}[theorem]{Problem}
\newtheorem{proposition}[theorem]{Proposition}
\newtheorem{corollary}[theorem]{Corollary}

\newtheorem{lemma}[theorem]{Lemma}
\newtheorem{claim}[theorem]{Claim}

\theoremstyle{definition}
\newtheorem{example}[theorem]{Example}

\newtheorem{remark}[theorem]{Remark}

\title[On $\kappa$-bounded and $M$-compact reflections of topological spaces]{On  $\kappa$-bounded and $M$-compact reflections\\ of topological spaces}
\author{Taras Banakh}
\address{T.Banakh: Ivan Franko National University of Lviv (Ukraine) and Jan Kochanowski University in Kielce (Poland)}
\email{t.o.banakh@gmail.com}
\begin{document}
\begin{abstract} For a topological space $X$ its reflection in a class $\Tau$ of topological spaces is a pair $(\Tau X,i_X)$ consisting of a space $\Tau X\in\Tau$ and continuous map $i_X:X\to \Tau X$ such that for any continuous map $f:X\to Y$ to a space $Y\in\Tau$ there exists a unique continuous map $\bar f:\Tau X\to Y$ such that $f=\bar f\circ i_X$. In this paper for an infinite cardinal $\kappa$ and a nonempty set $M$ of ultrafilters on $\kappa$, we study the reflections of topological spaces in the classes $\mathsf H_\kappa$ of $\kappa$-bounded Hausdorff spaces and $\HM$ of $M$-compact Hausdorff spaces (a topological space $X$ is $\kappa$-bounded if the closures of subsets of cardinality $\le\kappa$ in $X$ are compact; $X$ is {\em $M$-compact} if any function $x:\kappa\to X$  has a $p$-limit in $M$ for every ultrafilter $p\in M$).
\end{abstract}
\maketitle

\section{Introduction}

In this paper we shall describe the structure of reflections of topological spaces in some classes of Hausdorff compact-like spaces.

By a {\em reflection} of a topological space $X$ in a class $\Tau$ of topological spaces we understand a pair $(\Tau X,i_X)$ consisting of a space $\Tau X\in\Tau$ and a continuous map $i_X:X\to \Tau X$ such that for any continuous map $f:X\to Y$ to a topological space $Y\in\Tau$ there exists a unique continuous map $\bar f:\Tau X\to Y$ such that $f=\bar f\circ i_X$. The pair $(\Tau X,i_X)$ is called a {\em $\Tau$-reflection} of $X$.

The reflection $(\beta X,i_X)$ of a topological space $X$ in the class $\beta$ of compact Hausdorff spaces is known in General Topology \cite[\S3.6]{Eng} as the {\em Stone-\v Cech compactification} of $X$. It is well-known \cite[3.6.23]{Eng} that for a normal topological space $X$ its Stone-\v Cech compactification $\beta X$ can be identified with the Wallman compactification $\Wallman X$, which consists of ultrafilters of closed subsets in $X$.

The compactness is an important topological property, which have many (important) weakenings, see \cite{JSS},  \cite{GFK}, \cite{GR}, \cite{JV}. Let us recall some of them.

Let $\kappa$ be an infinite cardinal endowed with the discrete topology and $M\subset \beta\kappa$ be a nonempty set of ultrafilters on $\kappa$. 

A topological space $X$ is defined to be 
\begin{itemize}
\item {\em $\kappa$-bounded} if the closure $\bar A$ of any set $A\subset X$ of cardinality $|A|\le\kappa$ in $X$ is compact;
\item {\em $M$-compact} if for any ultrafilter $p\in M$ and function $x:\kappa\to X$ there exists a point $\bar x\in X$ such that $x^{-1}(U)\in p$ for any neighborhood $U\subset X$ of $\bar x$.
\end{itemize}
It is easy to see that 
\begin{itemize}
\item[\textup{(i)}] each compact space is $\kappa$-bounded;
\item[\textup{(ii)}] each $\kappa$-compact space is $M$-compact for any subset $M\subset\beta\kappa$;
\item[\textup{(iii)}] a topological space $X$ is compact if and only if it is $\beta|X|$-compact.
\end{itemize}

In this paper we shall reveal the structure of  reflections of topological spaces in the classes $\Hk$ of Hausdorff $\kappa$-compact spaces and $\HM$ of Hausdorff $M$-compact spaces. In Theorems~\ref{t:main-k} and \ref{t:main-M} we show that for a $\overline{\kappa}$-normal space $X$ (containing no long $\kappa$-butterflies) its $\Hk$-reflection (resp. $\HM$-reflection) can be realized as a subspace of the Wallman extension $\Wallman X$, endowed with a suitable topology. Theorems~\ref{t:main-k} and \ref{t:main-M} are proved in Sections~\ref{s:main-k} and \ref{s:main-M} after the necessary preliminary work made in Sections~\ref{s:Kak}--\ref{s:WM}.
 In particular, in Section~\ref{s:Kak}, using the classical approach of Kakutani \cite{Kak}, we prove Theorem~\ref{t:Kak} on the existence and uniqueness of reflections in productive closed-hereditary topological classes of Hausdorff topological spaces. In Section~\ref{s:Wallman} we recall the necessary information on the Wallman extensions; Section~\ref{s:kr+kn} is devoted to $\overline{\kappa}$-regular and $\overline{\kappa}$-normal spaces, introduced in \cite{BBR}. In Section~\ref{s:kb-bkc} we study the relations between $\kappa$-bounded and $\beta\kappa$-compact spaces. In Section~\ref{s:M-closed} we introduce $M$-closed sets and study their properties and their relation to $M$-compact spaces. In Section~\ref{s:M-continuous} we introduce $M$-continuous maps and study their continuity properties. In Section~\ref{s:WM} we introduce the Wallman $M$-compact extension $\WM X$ of a $T_1$-space $X$ and in Theorem~\ref{t:ex1} establish an important extension property of $\WM X$, which will be expoited in the proofs of the main Theorems~\ref{t:main-k} and \ref{t:main-M}  in Sections~\ref{s:main-k} and \ref{s:main-M}.

\section{The Hausdorff $\kappa$-bounded reflection of a topological space}\label{s:Kak}

A class $\Tau$ of topological spaces is called
\begin{itemize}
\item {\em topological} if for any topological space $X\in\Tau$ and any homeomorphism $h:X\to Y$ the topological space $Y$ belongs to the class $\Tau$;
\item {\em closed-hereditary} if for any topological space $X\in\Tau$, all closed subspaces of $X$ belong to the class $\Tau$;
\item {\em productive} if for any family $(X_\alpha)_{\alpha\in A}$ of topological spaces in the class $\Tau$ their Tychonoff product $\prod_{\alpha\in A}X_\alpha$ belongs to the class $\Tau$.
\end{itemize}

The following theorem should be known but we could not find the precise formulation in the literature.

\begin{theorem}\label{t:Kak} Let $\Tau$ be a productive closed-hereditary topological class of Hausdorff topological spaces. Every  topological space $X$ has a $\Tau$-reflection $(\Tau X,i_X)$. Moreover, the $\Tau$-reflection is unique in the sense that for any other $\Tau$-reflection $(\Tau'\!X,i'_X)$ of $X$ there exists a unique homeomorphism $h:\Tau X\to \Tau'\!X$ such that $h\circ i_X=i'_X$.
\end{theorem}

\begin{proof} In the proof we follow the classical idea of Kakutani \cite{Kak}. If $X=\emptyset$, then put $\Tau X=\emptyset$ and $i_X:X\to \Tau X$ be the unique map between the empty sets. So, assume that $X$ is not empty.

Consider the cardinal $\mu=2^{2^{|X|}}$. For any non-zero cardinal $\lambda\le\mu$ let $\mathcal T_\lambda$ be the family of all possible Hausdorff topologies on $\lambda$ such that the topological space $\lambda_\tau:=(\lambda,\tau)$ belongs to the class $\Tau$. It is clear that $|\mathcal T_\lambda|\le 2^{2^\lambda}$. For every topology $\tau\in\mathcal  T_\lambda$ let $\mathcal F_\tau$ be the family of all continuous functions from $X$ to $(\lambda,\tau)$. Now consider the Tychonoff product $\Pi:=\prod_{0<\lambda\le\mu}\prod_{\tau\in\mathcal T_\lambda}\lambda_\tau^{\F_\tau}$ and the diagonal map
$$i_X:X\to\Pi,\;\;i_X:x\mapsto \big(\big((f(x))_{f\in\F_\tau}\big)_{\tau\in\mathcal T_\lambda}\big)_{0<\lambda\le\mu}.$$
Let $\Tau X$ be the closure of $i_X(X)$ in the space $\Pi$. Since the class $\Tau$ is productive and closed-hereditary, the closed subspace $\Tau X$ of the Tychonoff product $\Pi$ of the spaces $\lambda_\tau\in\Tau$ belongs to the class $\Tau$.

It remains to show that the pair $(\Tau X,i_X)$ is a $\Tau$-reflection of $X$. Given any continuous map $g:X\to Y$ to a space $Y\in\Tau$, we need to find a unique continuous map $\bar g:\Tau X\to Y$ such that $g=\bar g\circ i_X$. The uniqueness of $\bar g$ follows from the density of $i_X(X)$ in $\Tau X$ and the Hausdorffness of $Y$. To show that the map $\bar g$ exists, consider the closure $\overline{g(X)}$ of $g(X)$ in $Y$. By \cite[1.5.3]{Eng}, $|\overline{g(X)}|\le 2^{2^{|g(X)|}}\le 2^{2^{|X|}}=\mu$.  Put $\lambda=|\overline{g(X)}|$ and take any bijection $h:\overline{g(X)}\to \lambda$. Endow $\lambda$ with the topology $\tau=\{h(U):U\subset \overline{g(X)}$ is open$\}$ and observe that $h:\overline{g(X)}\to(\lambda,\tau)$ is a homeomorphism. Since the class $\Tau$ is closed-hereditary and topological, the space $\lambda_\tau:=(\lambda,\tau)$ belongs to the class $\Tau$ and the topology $\tau$ to the family $\mathcal T_\lambda$. Consider the continuous map $f=h\circ g:X\to\lambda_\tau$ and let $\pr_f:\Pi\to \lambda_\tau$ be the coordinate projection. Then $\bar g:=h^{-1}\circ \pr_f{\restriction}{\Tau X}:\Tau X\to \overline{g(X)}\subset Y$ is a required continuous map such that $\bar g\circ i_X=h^{-1}\circ\pr_f\circ i_X=h^{-1}\circ f=h^{-1}\circ h\circ g=g$.
\smallskip

The uniqueness of the $\Tau$-reflection follows from the definition.
\end{proof}

\section{The Wallman extension of a $T_1$-space}\label{s:Wallman}

Theorem~\ref{t:Kak} is rather non-constructive and says nothing about the structure of  $\Tau$-reflections. In some cases, fortunately, there exist more informative ways of defining $\Tau$-reflections. Such cases include reflections of (normal) topological spaces in the class of compact Hausdorff spaces, which can be realized with the help of the  Wallman extension $\Wallman X$ of a $T_1$-space $X$. By a {\em $T_1$-space} we understand a topological space in which all finite subsets are closed.

We recall \cite[\S3.6]{Eng} that the {\em Wallman extension} $\Wallman X$ of a topological space $X$ consists of {\em closed ultrafilters}, i.e., families $\F$ of closed subsets of $X$ satisfying the following conditions:
\begin{itemize}
\item $\emptyset\notin\F$;
\item $A\cap B\in\F$ for any $A,B\in\F$;
\item a closed set $F\subset X$ belongs to $\F$ if $F\cap A\ne\emptyset$ for every $A\in\F$.
\end{itemize}
The Wallman extension $\Wallman X$ of $X$ carries the topology generated by the base consisting of the sets
$$\langle U\rangle=\{\F\in \Wallman X:\exists F\in\F\;\;(F\subset U)\}$$ where $U$ runs over open subsets of $X$.

%By Theorem 3.6.11 \cite{Eng}, for a discrete topological space $X$ of infinite cardinality $\kappa$ the space $\Wallman X$ has cardinality $2^{2^\kappa}$ and weight $2^\kappa$.

The proof of Theorem~\cite[3.6.21]{Eng} yields the following important fact.

\begin{proposition}\label{p:W-comp} The Wallman extension $\Wallman X$ of any topological space $X$ is compact.
\end{proposition}

A topological space $X$ is {\em normal} if for any disjoint closed sets $A,B\subset X$ there are disjoint open sets $U,V\subset X$ such that $A\subset U$ and $B\subset V$. The following characterization can be found in \cite[3.6.22]{Eng}.

\begin{proposition}\label{p:W-normal} A $T_1$-space $X$ is normal if and only if its Wallman extension $\Wallman X$ is Hausdorff.
\end{proposition}

If $X$ is a $T_1$-space, then we can consider the map $j_X:X\to \Wallman X$ assigning to each point $x\in X$ the principal closed ultrafilter consisting of all closed sets $F\subset X$ containing the point $x$. It is easy to see that the image $j_X(X)$ is dense in $\Wallman X$. By \cite[3.6.21]{Eng}, the map $j_X:X\to \Wallman X$ is a topological embedding. So, $X$ can be identified with the subspace $j_X(X)$ of $\Wallman X$.

The Wallman extension has the following extension property, proved in Theorem 3.6.21 in \cite{Eng}.

\begin{proposition}\label{p:W-ext} For any continuous map $f:X\to Y$ from a $T_1$-space $X$ to a compact Hausdorff space $Y$ there exists a unique continuous map $\bar f:\Wallman X\to Y$ such that $f=\bar f\circ j_X$.
\end{proposition}

Propositions~\ref{p:W-comp}, \ref{p:W-normal} and \ref{p:W-ext} imply the following known fact \cite[3.6.23]{Eng}.

\begin{corollary} For any normal $T_1$-space $X$ its Wallman extension $(\Wallman X,j_X)$ is a reflection of $X$ in the class  of compact Hausdorff spaces.
\end{corollary}

Therefore, for a normal $T_1$-space $X$ its  Stone-\v Cech compactification $(\beta X,i_X)$  can be identified with its Wallman extension $(\Wallman X,j_X)$. 

%Given an infinite cardinal $\kappa$, in the Wallman extension $\Wallman X$ of a topological $T_1$-space $X$, consider the subspace
%$$\Wk X:=\bigcup\{\overline{j_X(C)}:C\subset X,\;|C|\le\kappa\}.$$
%The subspace $\Wk X$ was introduced in \cite{BBR} and called the {\em Wallman $\kappa$-bounded extension} of $X$. By \cite{BBR}, the space $\Wk X$ is $\kappa$-bounded. By \cite{BBR}, the Hausdorff property of $\Wk X$ is equivalent to the $\overline{\kappa}$-normality of $X$. The latter property is discussed in the next section.

\section{$\overline{\kappa}$-Urysohn, $\overline{\kappa}$-regular and $\overline{\kappa}$-normal topological spaces}\label{s:kr+kn}

We recall that a topological space $X$ is {\em Urysohn} if any distinct points in $X$ have disjoint closed neighborhoods. 

Given an infinite cardinal $\kappa$, we define a topological space $X$ to be
\begin{itemize}
\item {\em $\overline{\kappa}$-Urysohn} if for any subset $C\subset X$ of cardinality $\le\kappa$ and distinct points $x,y\in X$ there are two disjoint open sets $V,W$ in $X$ such that $x\in V$, $y\in W$ and $\overline{V\cap C}\cap \overline{W\cap C}=\emptyset$;
\item {\em $\overline{\kappa}$-regular} if for any subset $C\subset X$ of cardinality $|C|\le\kappa$, any closed subset $F\subset\bar C$ and point $x\in X\setminus F$ there are disjoint open sets $V,W\subset X$ such that $x\in V$ and $F\subset W$;
\item {\em $\overline{\kappa}$-normal} if for any subset $C\subset X$ of cardinality $|C|\le\kappa$ and disjoint closed sets $A,B\subset\bar C$ there are disjoint open sets $V,W\subset X$ such that $A\subset V$ and $B\subset W$.
 \end{itemize}
 
It is easy to see that each $\overline{\kappa}$-Urysohn space is Hausdorff. 
 
\begin{lemma}\label{l:Ur} Each $\overline{\kappa}$-regular $T_1$-space $X$ is $\overline{\kappa}$-Urysohn.
\end{lemma}

\begin{proof} Fix any subset $C\subset X$ of cardinality $|C|\le\kappa$ and two distinct points $x,y\in X$. Being $\overline{\kappa}$-regular, the $T_1$-space $X$ is Hausdorff. Therefore, there are two disjoint open sets $U,W$ in $X$ such that $x\in U$ and $y\in W$. Observe that $F=\{y\}\cup \overline{W\cap C}$ is a closed subset of the set $\overline{\{y\}\cup C}$ and does not contain $x$. By the $\overline{\kappa}$-regularity of $X$, there are two disjoint open sets $V,W'$ in $X$ such that $x\in V$ and $F\subset W'$. Then $\overline{V\cap C}\cap \overline{W\cap C}\subset \overline{V}\cap W'=\emptyset$, witnessing that the space $X$ is $\overline{\kappa}$-Urysohn. 
\end{proof}

It is clear that each $\overline{\kappa}$-normal $T_1$-space is $\overline{\kappa}$-regular. The converse is true if all closed subets of density $\le\kappa$ are Lindel\"of. We recall that the {\em density} $d(X)$ of a topological space $X$ is the smallest cardinality of a dense set in $X$. A topological space $X$ is {\em Lindel\"of\/} if each open cover of $X$ contains a countable subcover.
The proofs of the following four propositions can be found in \cite{BBR}.

\begin{proposition}\label{p:rn} A $T_1$-space $X$ is $\overline{\kappa}$-normal if it is $\overline{\kappa}$-regular and each closed subset $C\subset X$ of density $d(C)\le\kappa$ is Lindel\"of.
\end{proposition}

\begin{proposition}\label{p:cn} Each $\kappa$-bounded Hausdorff space is $\overline{\kappa}$-normal and hence $\overline{\kappa}$-regular.
\end{proposition}

\begin{proposition} Each subspace of a $\overline{\kappa}$-regular space is $\overline{\kappa}$-regular.
\end{proposition}

\begin{proposition} The Tychonoff product of $\overline{\kappa}$-regular spaces is $\overline{\kappa}$-regular.
\end{proposition}

Similar properties hold for $\overline{\kappa}$-Urysohn spaces. The proofs of the following two propositions are standard and so are left as exercises to the reader.

\begin{proposition} Each subspace of a $\overline{\kappa}$-Urysohn space is $\overline{\kappa}$-Urysohn.
\end{proposition}

\begin{proposition} The Tychonoff product of $\overline{\kappa}$-Urysohn spaces is $\overline{\kappa}$-Urysohn.
\end{proposition}

%\begin{proposition}\label{p:WH} A $T_1$-space $X$ is $\overline{\kappa}$-normal if and only if its Wallman $\kappa$-bounded extension $\Wk X$ is Hausdorff.
%\end{proposition}

%\end{document}

\section{$\kappa$-Boundedness versus $\beta\kappa$-compactness}\label{s:kb-bkc}

%In this section we study the relation between  $\kappa$-boundedness and  $\beta\kappa$-compactness.

Let $\kappa$ be an infinite cardinal endowed with the discrete topology. Since the discrete space $\kappa$ is normal, its Wallman compactifiation $\Wallman \kappa$ can be identified with the Stone-\v Cech compactification $\beta\kappa$ of $\kappa$. Therefore, $\beta\kappa$ is a compact Hausdorff space consisting of all ultrafilters on $\kappa$. 

By a {\em $\kappa$-sequence} in a topological space $X$ we understand any function $x:\kappa\to X$, which will be also written as $(x_\alpha)_{\alpha\in\kappa}$. %The rangle $\{x(\alpha):\alpha\in\kappa\}$ of  $x$ will be denoted by $x[\kappa]$.

 Given an ultrafilter $p\in\beta\kappa$, we say that a $\kappa$-sequence $x:\kappa\to X$ is  {\em $p$-convergent} to a point $\bar x\in X$ if for any neighborhood $U$ of $\bar x$ in $X$ the set $x^{-1}(U)$  belongs to the ultrafilter $p$. In this case we  call the point $\bar x$  a {\em $p$-limit} of the $\kappa$-sequence $x$. By $\lim_p x$ we denote the set of all $p$-limit points of the sequence in $X$. If the space $X$ is Hausdorff, then the set $\lim_p x$ contains at most one point. %For a subset $M\subset\beta \kappa$ and a $\kappa$-sequence $x:\kappa\to X$ we put $\lim_M x:=\bigcup_{p\in M}\lim_p x$.

Let $M\subset\beta\kappa=\Wallman \kappa$ be a nonempty set of ultrafilters on $\kappa$. A topological space $X$ is defined to be
\begin{itemize}
\item {\em $M$-Hausdorff} if for any ultrafilter $p\in M$ and any $\kappa$-sequence $x:\kappa\to X$ the set $\lim_p x$ contains at most one point;
%\item {\em $p$-compact} for an ultrafilter $p\in\beta\kappa$, if any $\kappa$-sequence in $X$ has a $p$-limit point in $X$;
\item {\em $M$-compact}  if for any ultrafilter $p\in M$ and any $\kappa$-sequence $x:\kappa\to X$ the set $\lim_p x$ is not empty.
\end{itemize}

It is clear that each compact space is $M$-compact and each Hausdorff space is $M$-Hausdorff.

\begin{proposition}\label{p:MH=T1} Each $M$-Hausdorff space $X$ is a $T_1$-space.
\end{proposition}

\begin{proof} Given any point $x\in X$, we should prove that the singleton $\{x\}$ is closed in $X$. Take any ultrafilter $p\in M$ and consider the constant $\kappa$-sequence $s:\kappa\to\{x\}$. Observe that the set $\lim_p s$ coincides with the closure $\overline{\{x\}}$ of $\{x\}$ in $X$. Since $X$ is $M$-Hausdorff, the set $\lim_p s=\overline{\{x\}}$ is a singleton. Then the singleton $\{x\}=\overline{\{x\}}$ is closed in $X$.
\end{proof}

The following characterization generalizes Theorem~4.9 in \cite{JV}.

\begin{theorem}\label{t:kb=kr+bk} A $T_1$-space $X$ is Hausdorff and $\kappa$-bounded if and only if it is $\overline{\kappa}$-regular and $\beta\kappa$-compact.
\end{theorem}

\begin{proof} To prove the ``only if'' part, assume that $X$ is Hausdorff and $\kappa$-bounded. By Proposition~\ref{p:cn}, the space $X$ is $\overline{\kappa}$-regular. Assuming that $X$ is not $\beta\kappa$-compact, we can find an ultrafilter $p\in\beta\w$ and a $\kappa$-sequence $x:\kappa\to X$ such that $\lim_px=\emptyset$. Then for every $x\in X$ we can find an open neighborhood $O_x\subset X$ such that $x^{-1}(O_x)\notin p$ and hence $x^{-1}(X\setminus O_x)\in p$, by the maximality of the filter $p$. By the $\kappa$-boundedness of $X$, the rangle $x(\kappa)=\{x(\alpha)\}_{\alpha\in\kappa}$ of $x$ is contained in some compact subset $K$ of $X$. By the compactness of $K$, the open cover $\{O_x:x\in K\}$ of $K$ has a finite subcover $\{O_{x_1},\dots,O_{x_n}\}$. Take any $\alpha\in\bigcap_{i=1}^nx^{-1}(X\setminus O_{x_i})\in p$ and observe that $x(\alpha)\in\bigcap_{i=1}^n(X\setminus O_{x_i})=X\setminus\bigcup_{i=1}^nO_{x_i}=\emptyset$, which is a desired contradiction.
\smallskip

To prove the ``if'' part, assume that $X$ is $\overline{\kappa}$-regular and $\beta\kappa$-compact. The $\overline{\kappa}$-regularity of the $T_1$-space $X$ implies that $X$ is Hausdorff. Assuming that $X$ is not $\kappa$-bounded, we can find a set $\{x_\alpha\}_{\alpha\in\kappa}$ whose closure $Z$ in $X$ is not compact and hence admits an open cover $\U$ having no finite subcovers of $Z$. Let $[\U]^{<\w}$ be the set of all finite subfamilies of $\U$. By the choice of  $\U$, for every $\V\in[\U]^{<\w}$ the closed subset $Z\setminus\bigcup\V$ of $Z$ is non-empty.  Consider the filter $\F$ on $\kappa$ generated by the base consisting of the  sets $\{\alpha\in\kappa:x_\alpha\in U\}$ where $U$ is an open set in $X$ containing the closed set $Z\setminus \bigcup\V$ for some $\V\in[\U]^{<\w}$. Let $p\in\beta\kappa$ be any ultrafilter enlarging the filter $\F$. By the $\beta\kappa$-compactness of $X$, the $\kappa$-sequence $(x_\alpha)_{\alpha\in\kappa}$ has a $p$-limit $\bar x\in Z$. Find a set $U\in\U$ containing $\bar x$. By the $\overline{\kappa}$-regularity of $X$, there exist disjoint open sets $V,W$ in $X$ such that $\bar x\in V$ and $Z\setminus U\subset W$. By the definition of the filter $\F$, the set $F:=\{\alpha\in\kappa:x_\alpha\in W\}$ belongs to $\F\subset p$. On the other hand, by the $p$-convergence of $(x_\alpha)_{\alpha\in\kappa}$ to $\bar x\in V$, the set $E=\{\alpha\in \kappa:x_\alpha\in V\}$ also belongs to $p$. Then $\emptyset=E\cap F\in p$, which contradicts the choice of $p$ as a filter.
\end{proof}

%By \cite[??]{GFK}, for any nonempty subspace $M\subset\kappa$ the class $\HM$ of Hausdorff $M$-compact spaces is topological, closed-hereditary and productive. 
%Applying Theorem~\ref{t:Kak} to this class, we conclude that each topological space $X$ has a unique $\HM$-reflection $(\HM X,i_X)$, called the {\em Hausdorff $M$-compact reflection} of $X$. %Since each $\kappa$-bounded space is $\beta\kappa$-compact and each $\beta\kappa$-compact space is $M$-compact for any nonempty subset $M\subset\beta\kappa$, we obtain the natural continuous maps
%$$\HM X\to \mathsf H_{\beta\kappa} X\to {\mathsf H}_\kappa X.$$

%\begin{problem}\label{prob:HM} Given a topological space $X$, find a concrete model of its $\HM$-reflection.
%\end{problem}

%In Section~\ref{??} we shall answer this problem for $\overline{\kappa}$-normal $T_1$-spaces containing no $\kappa$-butterflies. But first we need study $M$-compact spaces in more details.

\section{$M$-closed sets in topological spaces}\label{s:M-closed}

In this section we assume that $\kappa$ is an infinite cardinal, endowed with the discrete topology, and $M$ is a nonempty subset of the Stone-\v Cech compactification $\beta\kappa=\Wallman \kappa$ of $\kappa$. 

 A subset $A$ of a topological space $X$ is defined to be  {\em $M$-closed in $X$} if $\lim_p x\subset A$ for any $\kappa$-sequence $x:\kappa\to A$ and any ultrafilter $p\in M$.
 
It is clear that the intersection of an arbitrary family of $M$-closed sets in $X$ is $M$-closed in $X$. The union of $M$-closed sets also is $M$-closed, but this is a less trivial fact.

\begin{lemma}\label{l:union} For any $M$-closed sets $A_1,A_2$ in a topological space $X$ the union $A_1\cup A_2$ is $M$-closed in $X$.
\end{lemma}

\begin{proof} To show that $A_1\cup A_2$ is $M$-closed, take any $\kappa$-sequence $x:\kappa\to A_1\cup A_2$ that $p$-converges to some point $\bar x\in X$ for some ultrafilter $p\in M$.  Since $p\ni \kappa=x^{-1}(A_1)\cup x^{-1}(A_2)$ is an ultrafilter, there exists $i\in\{1,2\}$ such that $x^{-1}(A_i)\in p$. Take any $\kappa$-sequence $y:\kappa\to A_i$ such that $y(\alpha)=x(\alpha)$ for all $\alpha\in x^{-1}(A_i)$ and observe that $\bar x\in \lim_p y$. The $M$-closedness of the set $A_i$ implies that $\bar x\in A_i\subset A_1\cup A_2$, witnessing that the union $A_1\cup A_2$ is $M$-closed.
\end{proof}

The following two lemmas show that the $M$-compactness has two typical properties of the compactness.

\begin{lemma}\label{l:Mc=>Mk} Each $M$-closed subspace $F$ of an $M$-compact space $X$ is $M$-compact.
\end{lemma}

\begin{proof}  To prove that $F$ is $M$-compact, take any ultrafilter $p\in M$ and any $\kappa$-sequence $x:\kappa\to F$. By the $M$-compactness of $X$ the $\kappa$-sequence $x$ is $p$-convergent to some point $\bar x\in X$ and by the $M$-closedness of $F$ in $X$, the point $\bar x$ belongs to the set $F$, witnessing that $A$ is $M$-compact.
\end{proof}

\begin{lemma}\label{l:Mk=>Mc} Each $M$-compact subspace $X$ of an $M$-Hausdorff space $Y$ is $M$-closed in $Y$.
\end{lemma}

\begin{proof} To prove that $X$ is $M$-closed in $Y$, we need to show that for any ultrafilter $p\in M$ and any $\kappa$-sequence $x:\kappa\to X$ the set $\lim_p x$ of $p$-limit points of $x$ in $Y$ is contained in $X$. By the $M$-compactness of $X$, the set $\lim_p x$ contains some point $\bar x\in X$. By the $M$-Hausdorff property of $Y$, the set $\lim_p x$ coincides with the singleton $\{\bar x\}$ and hence $\lim_p x=\{\bar x\}\subset X$, witnessing that the set $X$ is $M$-closed in $Y$.
\end{proof}

For a subset $A\subset X$ its {\em $M$-closure} $\cl_M A$ is the smallest $M$-closed subset of $X$ that contains $A$. The set $\cl_M A$ is equal to the intersection of all $M$-closed subsets of $X$ that contain $A$. It is easy to see that $$\cl_M A=\cl^{<\kappa^+}_M A=\bigcup_{\alpha<\kappa^+}\cl^\alpha_M A,$$
where $\cl_M^0 A=A$, $\cl^1_M A=\bigcup_{x\in A^\kappa}\bigcup_{p\in M}\lim_p x$, and for any ordinal $\alpha\ge 1$, 
$$\cl^\alpha A=\cl_M^1\big(\cl_M^{<\alpha} A),\mbox{ \ where \ }\cl_M^{<\alpha} A=\bigcup_{\gamma<\alpha}\cl_M^\gamma A.$$
%This description implies that $|\cl_M A|\le \max\{|M|^\kappa,|A|^\kappa\}$ if the space $X$ is Hausdorff.

\begin{lemma}\label{l:card} For any subset $A$ of an $M$-Hausdorff space $X$ its $M$-closure has cardinality $$|\cl_M A|\le \max\{|M|^\kappa,|A|^\kappa\}.$$
\end{lemma}

\begin{proof} Put $\lambda=\max\{|M|^\kappa,|A|^\kappa\}$ and observe that $|\lambda^\kappa|=\lambda$. If $\lambda=1$, then $A$ contains at most one point and by the $M$-Hausdorff property of $X$, $|\cl_M A|=|A|\le 1$. So, we assume that $\lambda\ge 2$. In this case $\kappa^+\le |2^\kappa|\le|\lambda^\kappa|=\lambda$.

Since $\cl_M A=\cl^{\kappa^+}_M A$, it suffices to prove that  $|\cl_M^\alpha A|\le\lambda$  for every $\alpha\le\kappa^+$. For $\alpha=0$ we have $|\cl_M^0A|=|A|\le\lambda$. 

Assume that for some ordinal $\alpha\le\kappa^+$ and all ordinals $\gamma<\alpha$ we have proved that $|\cl^\gamma_MA|\le\lambda$. Then $|\cl^{<\alpha}_M A|\le|\alpha|\cdot|\lambda|\le\kappa^+\cdot\lambda  \le\lambda\cdot\lambda=\lambda$. Taking into account that 
$$\textstyle{\cl^\alpha_M A=\cl^1_M(\cl^{<\alpha}_MA)=\bigcup\{\lim_p x:p\in M,\;x\in (\cl^{<\alpha}_MA)^\kappa\}},$$
we conclude that
$$|\cl^\alpha_M A|\le|M|\cdot|\cl^{<\alpha}_MA|^\kappa\le|M|\cdot |\lambda|^\kappa=\lambda.$$
\end{proof}

\begin{remark} By \cite[3.6.11]{Eng} the space $\beta\kappa=\Wallman \kappa$ has cardinality $2^{2^\kappa}$ and weight $2^\kappa$.
\end{remark}

Let $[M]$ be the $M$-closure of the set $\kappa$ in $\beta\kappa$. Lemma~\ref{l:card} implies that $|[M]|\le \max\{|M|^\kappa,|2^\kappa|\}$. Observe that any ultrafilter $p\in M$ coincides with the $p$-limit of the identity $\kappa$-sequence $\kappa\to\kappa$, which implies $M\subset[M]$. We shall show that the set $[M]$ is invariant under continuous self-maps of $\beta\kappa$ induced by self-maps of $\kappa$.
It is well-known that any function $f:\kappa\to\kappa$ can be uniquely extended to a continuous function $\bar f:\beta\kappa\to\beta\kappa$.

\begin{lemma}\label{l:Minv} For any function $f:\kappa\to\kappa$ we have $\bar f([M])\subset[M]$.
\end{lemma}

\begin{proof} Since $[M]=\cl^{<\kappa^+}_M \kappa$, it suffices to prove that $\bar f(\cl^\alpha_M \kappa)\subset[M]$ for any ordinal $\alpha$. For $\alpha=0$ this is trivial: $\bar f(\cl^0_M \kappa)=f(\kappa)\subset \kappa\subset[M]$. Assume that for some ordinal $\alpha$ and all ordinals $\gamma<\alpha$ we have proved that $\bar f(\cl^\gamma_M\kappa)\subset[M]$. Then also  $\bar f(\cl^{<\alpha}_M\kappa)\subset[M]$ where $\cl^{<\alpha}_M\kappa=\bigcup_{\gamma<\alpha}\cl^\gamma_M\kappa$.

Given any ultrafilter $\bar x\in\cl^\alpha_M\kappa=\cl^1_M(\cl^{<\alpha}_M\kappa)$, find an ultrafilter $p\in M$ and a $\kappa$-sequence $x:\kappa\to\cl^{<\alpha}_M\kappa$ such that $\bar x\in \lim_p x$. The continuity of the map $\bar f:\beta\kappa\to\beta\kappa$ ensures that $\bar f(\bar x)$ is a $p$-limit of the $\kappa$-sequence $\bar f\circ x:\kappa\to \bar f(\cl^{<\alpha}_M\kappa)\subset [M]$. Now the $M$-closedness of $[M]$ guarantees that $\bar f(\bar x)\in\bar f(\lim_p x)\subset\lim_p (\bar f\circ x)\subset [M]$.
\end{proof}

\begin{lemma}\label{l:[M]} If a topological space $X$ is $M$-compact, then it is $[M]$-compact.
\end{lemma}

\begin{proof} For every ordinal $\alpha$, consider the subspace $M_\alpha=\cl^\alpha_M M$ of $\beta M$. Since $[M]=M_{\kappa^+}$, it suffices to prove that the $M$-compact space $X$ is $M_\alpha$-compact for every $\alpha$.

For $\alpha=0$ we have $M_0=M$; therefore the $M_0$-compactness of $X$ follows from the $M$-compactness of $X$.

Assume that for some ordinal $\alpha$ and all ordinals $\gamma<\alpha$ we have proved that $X$ is $M_\gamma$-compact. To show that $X$ is $M_\alpha$-compact, fix any $\kappa$-sequence $x:\kappa\to X$ and any ultrafilter $p\in M_\alpha=\cl^1_M(M_{<\alpha})$ where $M_{<\alpha}=\bigcup_{\gamma<\alpha}M_\gamma$. 
Since $M_\alpha=\cl^1_M(M_{<\alpha})$, there exists an ultrafilter $q\in M$ and a $\kappa$-sequence $u:\kappa\to M_{<\alpha}$ such that $p\in \lim_q u$ in $\beta\kappa$. For every $\alpha\in\kappa$, consider the ultrafilter $u(\alpha)\in M_{<\alpha}$. By the inductive assumption, the $M_{<\alpha}$-compactness of $X$ guarantees that the $\kappa$-sequence $x$ has a $u(\alpha)$-limit point $y_\alpha\in \lim_{u(\alpha)}x$ in $X$.  By the $M$-compactness of $X$, the $\kappa$-sequence $y:\kappa\to X$, $y:\alpha\mapsto y_\alpha$, has a $q$-limit point $\bar x\in \lim_q y$. We claim that $\bar x\in \lim_p x$. Indeed, for any open neighborhood $U\subset X$ of $\bar x\in\lim_q y$, the set $Q=\{\alpha\in \kappa:y_\alpha\in U\}$ belongs to the ultrafilter $q$. For every $\alpha\in Q$ the inclusion $y_\alpha\in U\cap  \lim_{u(\alpha)}x$ imply that the set $G_\alpha=\{\gamma\in\kappa:x(\gamma)\in U\}$ belongs to the ultrafilter $u(\alpha)$. Then the set $$\{\gamma\in\kappa:x(\gamma)\in U\}\supset\bigcup_{\alpha\in Q}G_\alpha$$belongs to the ultrafilter $p\in \lim_q u$, which means that $\bar x\in\lim_px$. 
\end{proof}

\begin{lemma}\label{l:clM} For any subset $A$ of a topological space $X$ we have $$\cl_M A\subset \cl_{[M]}A=\cl^1_{[M]}A.$$ If the space $X$ is $[M]$-Hausdorff and $M$-compact, then $\cl_M A=\cl_{[M]}A=\cl^1_{[M]} A$.
\end{lemma}

\begin{proof} The inclusion $\cl_{[M]}^1 A\subset \cl_{[M]}A$ is trivial. To see that 
$\cl_{[M]}A\subset \cl_{[M]}^1A$, it suffices to check that the set $\cl_{[M]}^1A$ is $[M]$-closed in $X$. 

Take any ultrafilter $p\in[M]$ and any $\kappa$-sequence $x:\kappa\to\cl^1_{[M]}A$ that $p$-converges to some point $\bar x\in X$. By definition of the set $\cl^1_{[M]}A$, for every $\alpha\in\kappa$ the point $x(\alpha)\in\cl^1_{[M]}A$ is a $p_\alpha$-limit of some $\kappa$-sequence $x_\alpha:\kappa\to A$ for some ultrafilter $p_\alpha\in[M]$.

Let $b:\kappa\to\kappa\times\kappa$ be any bijective map.   For every $\alpha\in\kappa$ consider the function $f_\alpha:\kappa\to \kappa$, $f_\alpha:\gamma\mapsto b^{-1}(\alpha,\gamma)$, and let $\bar f_\alpha:\beta\kappa\to\beta\kappa$ be its continuous extension. By Lemma~\ref{l:Minv}, the ultrafilter $q_\alpha:=\bar f_\alpha(p_\alpha)$ belongs to the set $[M]$.

Observe that the functions $f_\alpha$, $\alpha\in\kappa$, have pairwise disjoint ranges in $\kappa$. So, we can define a $\kappa$-sequence $z:\kappa\to A$  such that $z\circ f_\alpha=x_\alpha$ for every $\alpha\in\kappa$. 
We claim that $x(\alpha)\in\lim_{q_\alpha}z$. Indeed, for any neighborhood $U_\alpha\subset X$ of $x(\alpha)\in\lim_{p_\alpha}x_\alpha$, the set $P_\alpha=\{\gamma\in\kappa:x_\alpha(\gamma)\in U_\alpha\}$ belongs to the ultrafilter $p_\alpha$. Then $f_\alpha(P_\alpha)\in \bar f_\alpha(p_\alpha)=q_\alpha\in[M]$. Now observe that 
$$
\{\beta\in\kappa:z(\beta)\in U_\alpha\}\supset \{b^{-1}(\alpha,\gamma):\gamma\in\kappa,\;x_\alpha(\gamma)\in U_\alpha\}= f_\alpha(P_\alpha)\in q_\alpha,
$$which means that $x(\alpha)\in\lim_{q_\alpha}z$.

Observe that the ultrafilter $$r:=\{R\subset \kappa:\{\alpha\in\kappa:R\in q_\alpha\}\in p\}$$
is the $p$-limit of the $\kappa$-sequence $q:\kappa\to[M]$, $q:\alpha\mapsto q_\alpha$, in the compact Hausdorff space $\beta\kappa$. The $M$-closedness of the set $[M]$ in $\beta\kappa$ ensures that  $r\in[M]$. 

We claim that $\bar x\in \lim_r z$. Indeed, for any open neighborhood $U\subset X$ of $\bar x\in \lim_p x$,  the set $P=\{\alpha\in\kappa:x(\alpha)\in U\}$ belongs to the ultrafilter $P$. For every $\alpha\in P$ we have $x(\alpha)\in U\cap\lim_{q_\alpha}z$, which implies $Q_\alpha:=\{\gamma\in\kappa:z(\gamma)\in U\}\in q_\alpha$. By definition of the ultrafilter $r$, the set $R=\bigcup_{\alpha\in P}Q_\alpha$ belongs to $r$.  Consequently, the set 
$$\{\gamma\in\kappa:z(\gamma)\in U\}\supset\bigcup_{\alpha\in P}Q_\alpha$$
belongs to $r$, which means that $\bar x\in \lim_r z\subset \cl^1_{[M]}A$. This completes the proof of the $[M]$-closedness of the set $\cl^1_{[M]}A$ and the equality $\cl^1_{[M]}A=\cl_{[M]}A$. Then $\cl_MA\subset\cl_{[M]}A=\cl^1_{[M]}A$.
\smallskip

Now assuming that the space $X$ is $M$-compact and $[M]$-Hausdorff, we shall prove that $\cl_{[M]}A\subset \cl_MA$. By Lemma~\ref{l:Mc=>Mk}, the $M$-closed subspace $\cl_M A$ of the $M$-compact space $X$ is $M$-compact and by Lemma~\ref{l:[M]}, the $M$-compact space $\cl_M A$ is $[M]$-compact. By Lemma~\ref{l:Mk=>Mc}, the $[M]$-compact subspace $\cl_M A$ of the $[M]$-Hausdorff space $X$ is $[M]$-closed in $X$. Then $\cl_{[M]}A\subset \cl_MA$ as $\cl_{[M]}A$ is the smallest $[M]$-closed subset of $X$ that contains $A$.
 \end{proof}
 
 \section{$M$-continuous maps between topological spaces}\label{s:M-continuous}
 
In this section we assume that $\kappa$ is an infinite cardinal endowed with the discrete topology, and $M$ is a nonempty subset of $\beta\kappa=\Wallman \kappa$. 
 
A function $f:Z\to Y$ between topological spaces is called {\em $M$-continuous} if $$\textstyle{f(\lim_p z)\subset \lim_p(f\circ z)}$$ for any $p\in M$ and $z\in Z^\kappa$. It is easy to see that each continuous function is $M$-continuous.

\begin{proposition} If a function $f:X\to Y$ between topological spaces is $M$-continuous, then for any $M$-closed set $B$ in $Y$ the preimage $f^{-1}(B)$ is $M$-closed in $X$.
\end{proposition}

\begin{proof} Let $B$ be an $M$-closed set in $Y$. To show that $f^{-1}(B)$ is $M$-closed in $X$, we need to check that $\lim_p x\subset f^{-1}(B)$ for any ultrafilter $p\in M$ and $\kappa$-sequence $x:\kappa\to f^{-1}(B)$. For the $\kappa$-sequence $f\circ x:\kappa\to B$, the $M$-closedness of the set $B$ implies $\lim_p (f\circ x)\subset B$, Since $f$ in $M$-continuous, $f(\lim_p x)\subset \lim_p(f\circ x)\subset B$ and hence $\lim_px\subset f^{-1}(B)$.
\end{proof}

\begin{lemma}\label{l:MMim} Let $f:X\to Y$ be a surjective $M$-continuous function between topological spaces. If the space $X$ is $M$-compact, then so is the space $Y$.
\end{lemma}

\begin{proof} To prove that $Y$ is $M$-compact, take any ultrafilter $p\in M$ and any $\kappa$-sequence $y:\kappa\to Y$. By the surjectivity of the function $f$, there exists a $\kappa$-sequence $x:\kappa\to X$ such that $y=f\circ x$. By the $M$-compactness of $X$ the set $\lim_p x$ is not empty. The $M$-continuity of $f$ ensures that $\emptyset\ne f(\lim_p x)\subset \lim_p(f\circ x)=\lim_p y$, which implies that the set $\lim_py$ is not empty and the space $Y$ is $M$-compact.
\end{proof}

Let us recall that a {\em tightness} $t(X)$ of a topological space $X$ is the smallest cardinal $\lambda$ such that for any subset $A$ in $X$ and any point $x\in\bar A$ there exists a subset $B\subset A$ of cardinality $|B|\le \lambda$ such that $x\in \bar B$.  

\begin{proposition} For any $\beta\kappa$-continuous function $f:X\to Y$ from a topological space $X$ to a topological space $Y$ and any subset $A\subset X$ of tightness $t(A)\le\kappa$ the restriction $f{\restriction}A$ is continuous.
\end{proposition}

\begin{proof} Assuming that $f{\restriction}A$ is discontinuous, we can find a point $a\in A$ and an open neighborhood $U\subset Y$ of $f(a)$ such that the set $A\setminus f^{-1}(U)$ contains $a$ in its closure. Since $t(A)\le\kappa$, there exists a set $B\subset A\setminus f^{-1}(U)$ of cardinality $|B|\le\kappa$ that contains the point $a$ in its closure. Let $x:\kappa\to B$ be any surjective map and $p\in\beta\kappa$ be an ultrafilter containing the sets $x^{-1}(O_a)$ where $O_a$ runs over neighborhoods of $a$ in $X$. The choice of $p$ guarantees that $a\in\lim_p x$. Then $f(a)\in \lim_p (f\circ x)$ by the $\beta\kappa$-continuity of $f$. Since $f(a)\in U\cap \lim_p(f\circ x)$, the set $P=\{\alpha\in \kappa:f(x(\alpha))\in U\}$ belongs to the ultrafilter $p$ and hence is not empty. On the other hand, $f(x(P))\subset f(x(\kappa))\subset f(B)\subset Y\setminus U$, which contradicts $f(x(P))\subset U$.
This contradiction completes the proof of the continuity of the restriction $f{\restriction}A$. 
\end{proof}

\begin{proposition}\label{p:cont-d} For any $\beta\kappa$-continuous function $f:X\to Y$ from a topological space $X$ to a $\overline{\kappa}$-regular topological space $Y$ and any subset $A\subset X$ of density $d(A)\le\kappa$ the restriction $f{\restriction}A$ is continuous.
\end{proposition}

\begin{proof} Assuming that the restriction $f{\restriction}A$ is discontinuous, 
we can find an open subset $U\subset Y$ such that the set $A\cap f^{-1}(U)$ is not open in the subspace topology of $A$. Consequently, there exists a point $a\in A\cap f^{-1}(U)$ whose any neighborhood $O_a\subset X$ intersects the set $A\setminus f^{-1}(U)$.

Fix a dense subset $D\subset A$ of cardinality $|D|=d(A)\le\kappa$ and find a surjective map $x:\kappa\to D$. By the $\overline{\kappa}$-regularity of $Y$, there are two disjoint open sets $V,W$ in $Y$ such that $f(a)\in V\subset U$ and $\overline{f(D)}\setminus U\subset W$. 

\begin{claim}\label{cl:q-well} For any open neighborhood $O_a\subset A$ of $a$, the preimage $x^{-1}(O_a\cap f^{-1}(W))$ is not empty.
\end{claim}

\begin{proof} By the choice of $U$, the set $O_a\setminus f^{-1}(U)$ contains some point $\bar z\in A$.
Let $p\in\beta\kappa$ be any ultrafilter containing the sets $x^{-1}(O_{\bar z})$ where $O_{\bar z}$ runs over neighborhoods of $\bar z$ in $A$. The density of the set $D=x(\kappa)$ in $A\ni\bar z$ ensures that the ultrafilter $p$ is well-defined. The definition of $p$ guarantees that $\bar z\in \lim_p x$. The $\beta\kappa$-continuity of $f$ ensures that  $f(\bar z)\in\lim_p(f\circ x)$. Consequently, $f(\bar z)\in \overline{f(D)}\setminus U\subset W$ and the set $P=\{\alpha\in\kappa:f\circ x(\alpha)\in W\}$ belongs to the ultrafilter $p$. Since $\bar z\in O_a\cap\lim_p x$, the set $P'=\{\alpha\in \kappa:x(\alpha)\in O_a\}$ belongs to $p$, too. Choose any  $\alpha\in P\cap P'\in p$ and conclude that $x(\alpha)\in O_a\cap f^{-1}(W)$ witnessing that $x^{-1}(O_a\cap f^{-1}(W))\ne \emptyset$.
\end{proof}

Let $q\in\beta\kappa$ be any ultrafilter containing the sets $x^{-1}(O_a\cap f^{-1}(W))$ where $O_a$ runs over neighborhoods of $a$ in $A$. Claim~\ref{cl:q-well} ensures that the ultrafilter $q$ is well-defined. The definition of $q$ guarantees that $a\in\lim_q x$. Then $f(a)\in\lim_q(f\circ x)$, by the $\beta\kappa$-continuity of $f$. Since $f(a)\in V\cap\lim_q(f\circ x)$, the set $Q=\{\alpha\in\kappa:f\circ x(\alpha)\in V\}=x^{-1}(f^{-1}(V))$ belongs to the ultrafilter $q$. On the other hand, the definition of $q$ guarantees that $x^{-1}(f^{-1}(W))\in q$. Then $\emptyset=x^{-1}(f^{-1}(V))\cap x^{-1}(f^{-1}(W))\in q$, which contradicts the definition of a filter. This contradiction completes the proof of the continuity of the restriction $f{\restriction}A$.
\end{proof}

\section{The Wallman $M$-compact extension of a $T_1$-space}\label{s:WM}
 
In this section we assume that $\kappa$ is an infinite cardinal endowed with the discrete topology, and $M$ is a nonempty subset of $\beta\kappa=\Wallman \kappa$. By $[M]$ we denote the $M$-closure of $\kappa$ in the compact Hausdorff space $\beta\kappa$. Lemma~\ref{l:clM} implies that $$[M]=\cl_M\kappa=\cl_M(\cl_M\kappa)=\cl_{[M]}(\cl_M\kappa)=\cl_{[M]}{[M]}.$$

Given any $T_1$-space $X$, consider its Wallman extension $(\Wallman X,j_X)$. Since $j_X:X\to \Wallman X$ is a topological embedding, we can identify $X$ with the subspace $j_X(X)$ of $\Wallman X$. Denote by $\WM X$ the $[M]$-closure $\cl_{[M]}X$ of the set $X=j_X(X)$ in $\Wallman X$. By Lemma~\ref{l:clM}, $$\WM X=\cl_{[M]}X=\cl^1_{[M]}X.$$
The space $\WM X$ is called {\em the Wallman $M$-compact extension} of the $T_1$-space. The following proposition justifies the choice of the terminology.
%We recall that
%$$\Wk X=\bigcup\{\overline{C}:C\subset X,\;|C|\le\kappa\}\subset \Wallman X,$$where the closure $\overline{C}$ of $C$ is taken in $\Wallman X$.

\begin{proposition}\label{p:clM} For any $T_1$-space $X$ the subspace $\WM X$ of  $\Wallman X$ is $[M]$-compact and $\WM X\subset \Wallman_{\!\beta\kappa} X\subset \Wallman X$.
\end{proposition}

\begin{proof} By Proposition~\ref{p:W-comp}, the space $\Wallman X$ is compact and hence $[M]$-compact. Then its $[M]$-closed subset $\WM X=\cl_{[M]}X$ is $[M]$-compact by  Lemma~\ref{l:Mc=>Mk}. The inclusion $[M]\subset \beta\kappa$ implies $\WM X\subset \Wallman_{\!\beta\kappa}X\subset\Wallman X$.
\end{proof}

\begin{proposition}\label{p:Wbk} For any $T_1$-space $X$ its Wallman $\beta\kappa$-compact extension $\Wallman_{\!\beta\kappa} X$ is $\kappa$-bounded and is equal to the subset $$\Wk X=\textstyle{\bigcup}\{\overline{j_X(C)}:C\subset X,\;|C|\le\kappa\}$$of $\Wallman X$.
\end{proposition}

\begin{proof} To see that $\Wallman_{\!\beta\kappa} X\subset \Wallman_{\!\bar \kappa} X$, it suffices to show that the subspace $\Wallman_{\!\bar \kappa} X$ is $\beta\kappa$-closed in $\Wallman X$.
Fix any ultrafilter $p\in\beta\kappa$ and any $\kappa$-sequence $x:\kappa\to \Wallman_{\!\bar\kappa} X$ that $p$-converges to some element $\bar x\in \Wallman X$. By the definition of $\Wallman_{\!\bar\kappa} X$, for any $\alpha\in\kappa$, there exists a set $C_\alpha\subset X$ of cardinality $|C_\alpha|\le\kappa$ such that $x(\alpha)\in\overline{j_X(C_\alpha)}\subset\Wallman X$. Observe that the set $C=\bigcup_{\alpha\in\kappa}C_\alpha$ has cardinality $|C|\le\kappa$. Then $\bar x\in \overline{C}\subset \Wk X$ by the definition of $\Wk X$.

To see that $\Wk X\subset\Wallman_{\!\beta\kappa}X$, take any element $\bar x\in \Wk X$ and find a subset $C\subset X$ of cardinality $|C|\le\kappa$ such that $\bar x\in \overline{j_X(C)}$. It follows that the set $j_X(C)$ is not empty and hence admits a surjective map $x:\kappa\to j_X(C)$. Denote by $w$ the topology of the Wallman extension $\Wallman X$ of $X$. Take any ultrafilter $p\in\beta\kappa$ containing the filter $\{x^{-1}(U):\bar x\in U\in w\}$ and observe that the point $\bar x$ is a $p$-limit of the $\kappa$-sequence $x:\kappa\to j_X(X)$ in $\Wallman X$. Then $\bar x\in\cl_{\beta\kappa} X=\Wallman_{\!\beta\kappa}X$ by the definition of $\cl_{\beta\kappa} X$. Therefore, $\Wallman_{\!\beta\kappa} X=\Wk X$.

To see that the space $\Wallman_{\!\beta\kappa}X=\Wk X$ is $\kappa$-bounded, take any subset $D\subset \Wk X$ of cardinality $|D|\le\kappa$. For any $u\in D\subset\Wk X$ find a subset $C_u\subset X$ of cardinality $\le\kappa$ such that $u\in\overline{j_X(C_u)}$. Then the set $C=\bigcup_{u\in D}C_u$ has cardinality $\le\kappa$ and $D\subset \overline{j_X(C)}$. The compactness of the Wallman extension $\Wallman X$ implies the compactness of $\overline{j_X(C)}$. Then the closure $\overline{D}$ of $D$ is compact, being a closed subset of the compact space $\overline{j_X(C)}$.
\end{proof}

%The proof of the following proposition can be found in \cite{BBR}.

\begin{proposition}\label{p:WH} For a $T_1$-space $X$ the following conditions are equivalent:
\begin{enumerate}
\item $X$ is $\overline{\kappa}$-normal;
\item $\Wallman_{\!\beta\kappa}X=\Wk X$ is Hausdorff;
\item $\Wallman_{\!\beta\kappa}X=\Wk X$ is $\overline{\kappa}$-Urysohn.
\end{enumerate}
\end{proposition}

\begin{proof} The equivalence $(1)\Leftrightarrow(2)$ was proved in \cite{BBR} and $(3)\Ra(2)$ is trivial. It remains to prove that $(2)\Ra(3)$. Assume that the space $\Wallman_{\!\beta\kappa}X=\Wk X$ is Hausdorff. By Propositions~\ref{p:Wbk}, \ref{p:cn}, and Lemma~\ref{l:Ur}, this space is $\kappa$-bounded, $\overline{\kappa}$-regular, and $\overline{\kappa}$-Urysohn. 
\end{proof}

Let us recall that a topological space $X$ is {\em extremally disconnected} if the closure of any open set is open. It is well-known that a topological space is extremally disconnected if and only if it contains no {\em butterflies}, i.e., pairs $(U,V)$ of disjoint open sets  such that $\overline{U}\cap\overline{V}\ne\emptyset$. A butterfly $(U,V)$ in $X$ is called a ({\em long}) {\em $\kappa$-butterfly} if there exist subsets $A\subset U$, $B\subset V$  of cardinality $\le\kappa$ such that $\overline{A}\cap\overline{B}\ne\emptyset$ (and the space $\overline{A}\cap\overline{B}$ is not compact).

\begin{theorem}\label{t:ex1} Let $X$ be a $T_1$-space and $Y$ be a Hausdorff $M$-compact space. Assume that either the space $Y$ is $\overline{\kappa}$-Urysohn or the space $X$ contains no long $\kappa$-butterflies.
Then for any continuous function $f:X\to Y$ there exists a unique $[M]$-continuous function $\bar f:\WM X\to Y$ such that $f=\bar f{\restriction}X$.
\end{theorem}

\begin{proof} We identify $X$ with its image $j_X(X)$ in $\Wallman X$. By Lemma~\ref{l:clM}, $\WM X=\cl_{[M]} X=\cl_{[M]}^1X$. 
Given any point $\bar x\in \WM X=\cl_{[M]}^1 X$, consider the family $\Pi_{\bar x}=\{(p,x)\in [M]\times X^\kappa:\bar x\in \lim_p x\}$.
For every $(p,x)\in \Pi_{\bar x}$, let $\lim_p (f\circ x)$ be the set of $p$-limit points of the $\kappa$-sequence $f\circ x$ in $Y$. By the Hausdorff property of $Y$, the set $\lim_p (f\circ x)$ contains at most one point. By Lemma~\ref{l:[M]}, the $M$-compact space $Y$ is $[M]$-compact, which implies that the set $\lim_p(f\circ x)$ is not empty and hence is a singleton. 

\begin{claim}\label{cl:coincide} For any $(p,x),(q,y)\in \Pi_{\bar x}$ the singletons $\lim_p (f\circ x)$ and $\lim_q (f\circ y)$ coincide. 
\end{claim}

\begin{proof} To derive a contradiction, assume that   $\lim_p (f\circ x)\ne \lim_q (f\circ y)$. Since $Y$ is Hausdorff, there are two disjoint open sets $V,W$ in $Y$ such that $\lim_p (f\circ x)\subset V$ and $\lim_q (f\circ y)\subset W$. 
Then $P:=\{\alpha\in\kappa:x(\alpha)\in f^{-1}(V)\}\in P$ and $Q:=\{\alpha\in\kappa:y(\alpha)\in f^{-1}(W)\}\in q$. The inclusion $\bar x\in\lim_p x\cap \lim_q y$ implies that the sets $\overline{x(P)},\overline{y(Q)}$ belong to the ultrafilter $\bar x$ and hence the set $E=\overline{x(P)}\cap \overline{y(Q)}\in \bar x$ is not empty. This means that $(f^{-1}(V),f^{-1}(W))$ is a $\kappa$-batterfly in $X$. 

Now we consider two cases.
\smallskip

1. The space $X$ contains no long $\kappa$-butterflies. In this case the set $E=\overline{x(P)}\cap\overline{y(Q)}\in\bar x$ is compact (otherwise, the pair $(f^{-1}(V),f^{-1}(W))$ would be a long butterfly in $X$). The compactness of $E\in \bar x$ guarantees that the set $\bigcap\bar x\subset E$ contains some point and the ultrafilter $\bar x$ coincides with the principal ultrafilter supported by this  point.
 Then $\lim_p(f\circ x)=\{f(\bar x)\}=\lim_q(f\circ y)$ by the Hausdorff property of $Y$ and continuity of $f$.
\smallskip

2. The space $Y$ is $\overline{\kappa}$-Urysohn. In this case we can replace $V,W$ 
by smaller open sets and assume that $\overline{V\cap (f\circ x(\kappa))}\cap \overline{W\cap (f\circ y(\kappa))}=\emptyset$. Then $f(E)\subset f(\overline{x(P)})\cap f(\overline{y(Q)})\subset \overline{f\circ x(P)}\cap \overline{f\circ y(Q)}\subset \overline{V\cap (f\circ x(\kappa))}\cap\overline{W\cap (f\circ y(\kappa))}=\emptyset$ and $\emptyset= E\in\bar x$, which contradicts the definition of a filter.
\end{proof}

Claim~\ref{cl:coincide} allows us to define the function $\bar f:\WM X\to Y$ assigning to each element $\bar x\in\WM X$ the unique element of the singleton $\bigcup_{(p,x)\in\Pi_{\bar x}}\lim_p x$.

\begin{claim} $\bar f{\restriction}X=f$.
\end{claim}

\begin{proof} For any $\bar x\in X$, consider the constant $\kappa$-sequence $x:\kappa\to\{\bar x\}$ and take any $p\in M$. Then $\bar x\in \lim_p x$ and $\{\bar f(\bar x)\}=\lim_p(f\circ x)=\{f(\bar x)\}$ as the $\kappa$-sequence $f\circ x:\kappa\to \{f(\bar x)\}$ is constant. Consequently, $\bar f(\bar x)=f(\bar x)$.
\end{proof}

\begin{claim} The function $\bar f$ is $[M]$-continuous. 
\end{claim}

\begin{proof} Given any ultrafilter $p\in[M]$ and  $\kappa$-sequence $x:\kappa\to\WM X$ that $p$-converges to a point $\bar x\in\WM X$, we should prove that $\bar f(\bar x)\in\lim_p (\bar f\circ x)$.

Since the space $Y$ is Hausdorff and $[M]$-compact (see Lemma~\ref{l:[M]}), the set $\lim_p(\bar f\circ x)$ is not empty and contains a unique point $\bar y$. We should prove that $\bar y=\bar f(\bar x)$. To derive a contradiction, assume that $\bar y\ne\bar f(\bar x)$.% Assuming that $\bar f(\bar x)\ne\bar y$, we can find two disjoint open sets $V,W$ in $Y$ such that $\bar f(\bar x)\in V$ and $\bar y\in W$.

For every $\alpha\in\kappa$, consider the ultrafilter $x(\alpha)\in \WM X=\cl_{[M]}^1X$ and find an ultrafilter $p_\alpha\in [M]$ and a $\kappa$-sequence $x_\alpha:\kappa\to X$ such that $x(\alpha)\in\lim_{p_\alpha}x_\alpha$. Let $b:\kappa\to\kappa\times\kappa$ be any bijection. For every $\alpha\in\kappa$ consider the function $h_\alpha:\kappa\to\kappa$, $h_\alpha:\gamma\mapsto b^{-1}(\alpha,\gamma)$.
Let $\bar h_\alpha:\beta\kappa\to\beta\kappa$ be the continuous extension of the function $h_\alpha$. By Lemma~\ref{l:Minv}, the ultrafilter $u_\alpha=\bar h_\alpha(p_\alpha)$ belongs to the set $\bar h_\alpha([M])\subset[M]$. 

Consider the $\kappa$-sequence $u:\kappa\to [M]$, $u:\alpha\mapsto u_\alpha$. Since the space $\beta\kappa$ is compact and Hausdorff, the set $\lim_p u$ contains a unique ultrafilter $r$, which belongs to the set $$\cl_{[M]}([M])=\cl_M([M])=\cl_M(\cl_M\kappa)=\cl_M\kappa=[M],$$
see Lemma~\ref{l:clM}.

 Since the functions $h_\alpha$, $\alpha\in\kappa$, have disjoint ranges, we can define a $\kappa$-sequence $z:\kappa\to X$ such that $z\circ h_\alpha(\gamma)=x_\alpha(\gamma)$ for every $\gamma\in \kappa$. 

\begin{claim}\label{cl:xa=luz} For every $\alpha\in\kappa$, \ $x(\alpha)\in \lim_{u_\alpha}z.$
\end{claim}

\begin{proof} Given any neighborhood $U\subset X$ of $x(\alpha)$, we should prove that the set $\{\beta\in\kappa:z(\beta)\in U\}$ belongs to the ultrafilter $u_\alpha$. Since $x(\alpha)\in U\cap\lim_{p_\alpha}x_\alpha$, the set $P=\{\gamma\in\kappa:x_\alpha(\gamma)\in U\}$ belongs to the ultrafilter $p_\alpha$. 

By the definition of the function $z$, for any $\gamma\in\kappa$ we have $z(h_\alpha(\gamma))=z(b^{-1}(\alpha,\gamma))=x_\alpha(\gamma)$. Consequently,
$$
\begin{aligned}
\{\beta\in\kappa:z(\beta)\in U\}\supset&\;\{h_\alpha(\gamma):\gamma\in\kappa,\;x_\alpha(\gamma)\in U\}=\\
&f_\alpha(\{\gamma\in\kappa:x_\alpha(\gamma)\in U\})=h_\alpha(P_\alpha)\in\bar f_\alpha(p_\alpha)=u_\alpha.
\end{aligned}
$$
\end{proof}  

\begin{claim}\label{cl:px=rz} $\bar x\in\lim_r z$.
\end{claim}

\begin{proof} Given any open neighborhood $U\subset X$ of $\bar x$, we need to prove that the set $\{\gamma\in\kappa:z(\gamma)\in U\}$ belongs to the ultrafilter $r$.

Since $\bar x\in U\cap\lim_p x$, the set $P=\{\alpha\in\kappa:x(\alpha)\in U\}$ belongs to the ultrafilter $p$. For every $\alpha\in P$ we have $x(\alpha)\in U\cap\lim_{u_\alpha}z$ (see Claim~\ref{cl:xa=luz}) and hence $U_\alpha:=\{\gamma\in\kappa:z(\gamma)\in V\}\in u_\alpha$. By the definition of the ultrafilter $r\in\lim_p u_\alpha$, the set $\bigcup_{\alpha\in P}U_\alpha$ belongs to $r$.
Then the set $\{\gamma\in\kappa:z(\gamma)\in U\}\supset \bigcup_{\alpha\in P}U_\alpha$ belongs to the ultrafilter $r$, too.
\end{proof} 

By Claim~\ref{cl:px=rz}, $(r,z)\in \Pi_{\bar x}$. Then $\lim_r(f\circ z)=\{\bar f(\bar x)\}\ne \{\bar y\}$. By the Hausdorff property of $Y$, there are two disjoint open sets $V,W$ in $Y$ such that $\lim_r(f\circ z)\subset V$ and $\lim_p(\bar f\circ x)=\{\bar y\}\subset W$. Then the set $R:=\{\gamma\in\kappa:f(z(\gamma))\in V\}$ belongs to the ultrafilter $r$ and the set $P=\{\alpha\in\kappa:\bar f(x(\alpha))\in W\}$ belongs to $p$. By the definition of the ultrafilter $r\in \lim_p u\in\langle R\rangle$, the set $P'=\{\alpha\in \kappa:R\in u_\alpha\}=\{\alpha\in\kappa:u_\alpha\in\langle R\rangle\}$ belongs to the ultrafilter $p$. Fix any ordinal $\alpha\in P\cap P'\in p$. By Claim~\ref{cl:xa=luz},  $x(\alpha)\in \lim_{u_\alpha} z$. Consequently, $(u_\alpha,z)\in\Pi_{x(\alpha)}$ and $\bar f(x(\alpha))\in \lim_{u_\alpha}(f\circ z)$ by the definition of $\bar f$. Now $\alpha\in P$ ensures that $\bar f(x(\alpha))\in W\cap \lim_{u_\alpha}(f\circ z)$ and hence the set $U_\alpha=\{\gamma\in\kappa:f\circ z(\gamma)\in W\}$ belongs to the ultrafilter $u_\alpha$. On the other hand, the inclusion $\alpha\in P'$ guarantees that $R\in u_\alpha$. Take any ordinal $\gamma\in U_\alpha\cap R$. Then $f(z(\gamma))\in W\cap V=\emptyset$, which is a desired contradiction finishing the proof of the equality $\bar f(\bar x)=\bar y$ and the proof of the $[M]$-continuity of the function $\bar f$. 
\end{proof}

Finally, we show that the function $\bar f$ is unique. Assume that $\tilde f:\WM X\to Y$ is another $[M]$-continuous function extending the function $f$. Given any point $\bar x\in\WM X$, we should prove that $\bar f(\bar x)=\tilde f(\bar x)$. Since $\bar x\in\WM X=\cl_{[M]}^1X$, there exists an ultrafilter $p\in[M]$ and a $\kappa$-sequence $x:\kappa\to X$ such that $\bar x\in\lim_p x$. The $[M]$-continuity of the functions $\bar f$ and $\tilde f$ guarantees that $\{\bar f(\bar x),\tilde f(\bar x)\}\subset \lim_p(f\circ x)$. Since $Y$ is Hausdorff the set $\lim_p(f\circ x)$ contains at most one point and hence $\bar f(\bar x)=\tilde f(\bar x)$.
\end{proof}

\begin{corollary}\label{c:rex} For any continuous function $f:X\to Y$ from a $T_1$-space $X$ to a regular $M$-compact $T_1$-space $Y$ there exists a unique continuous function $\bar f:\WM X\to Y$ such that $f=\bar f{\restriction}X$.
\end{corollary}

\begin{proof} By Theorem~\ref{t:ex1}, there exists an $[M]$-continuous function $\bar f:\WM X\to Y$ such that $f=\bar f\circ j_X$. To check the continuity of $f$, take any ultrafilter $\bar x\in \WM X$ and any open neighborhood $U\subset Y$ of $\bar f(\bar x)$. By the regularity of $Y$, there exist open sets $V$ and $W$ in $Y$ such that  $\bar f(\bar x)\in V\subset\overline{V}\subset W\subset\overline{W}\subset U$. By Proposition~\ref{l:clM}, $\WM X=\cl_{[M]}X=\cl_{[M]}^1X$. Consequently, we can find an ultrafilter $p\in[M]$ and a $\kappa$-sequence $x:\kappa\to X$ such that $\bar x\in \lim_p x$. The $[M]$-continuity of $\bar f$ ensures that $\bar f(\bar x)\in\lim_p(f\circ x)$. It follows from $\bar f(\bar x)\in V\cap\lim_p(f\circ x)$ that the set $P=\{\alpha\in \kappa:f(x(\alpha))\in V\}$ belongs to the ultrafilter $p$. Then $x(P)\subset f^{-1}(V)$ and $\bar x\in \lim_p x$ imply $\overline{f^{-1}(V)}\in \bar x$ and hence $\bar x\in \langle f^{-1}(W)\rangle$. It remains to prove that $\bar f(\langle f^{-1}(W)\rangle)\subset \overline{W}\subset U$. Take any ultrafilter $\bar z\in \langle f^{-1}(W)\rangle$. Since $\WM X=\cl_{[M]}^1X$, there exists an ultrafilter $q\in[M]$ and a $\kappa$-sequence $z:\kappa\to X$ such that $\bar z\in\lim_q z$. Since $\bar z\in\langle f^{-1}(W)\rangle\cap \lim_q z$, the set $Q=\{\alpha\in\kappa:z(\alpha)\in f^{-1}(W)\}$ belongs to the ultrafilter $q$. The $[M]$-continuity of $\bar f$ ensures that $\bar f(\bar z)\in\lim_q(f\circ z)\subset \overline{f(z(Q))}\subset \overline{f(f^{-1}(W))}\subset\overline{W}\subset U$. Therefore, the map $\bar f$ is continuous. The uniqueness of $\bar f$ follows from the density of $X=j_X(X)$ in $\WM X$.
\end{proof} 

\begin{corollary} For any normal $T_1$-space $X$ the pair $(\WM X,j_X)$ is a reflection of $X$ in the class  of regular $M$-compact $T_1$-spaces.
\end{corollary}

\begin{proof} By Proposition~\ref{p:W-normal}, the Wallman extension $\Wallman X$ of the normal $T_1$-space $X$ is a compact Hausdorff space. Then the subspace $\WM X$ of $\Wallman X$ is Tychonoff and hence regular. By Proposition~\ref{p:clM}, the space $\WM X$ is $M$-compact. Now Corollary~\ref{c:rex} implies that  $(\WM X,j_X)$ is a reflection of $X$ in the class of regular $M$-compact $T_1$-spaces.
\end{proof}

\section{The Hausdorff $\kappa$-bounded reflection of a topological space}\label{s:main-k}

Let $\kappa$ be an infinite cardinal. It is easy to see that the class $\mathsf H_k$ of Hausdorff $\kappa$-bounded spaces is productive, closed-hereditary and topological. 

By Theorem~\ref{t:Kak}, each topological space $X$ has a unique $\mathsf H_\kappa$-reflection $(\mathsf H_\kappa X,i_X)$. In this section we show that for a $\overline{\kappa}$-normal $T_1$-space $X$ its $\mathsf H_\kappa$-reflection can be identified with the Wallman $\beta\kappa$-compact extension $\Wallman_{\!\beta\kappa}X$ endowed with a suitable topology $\tau^\star$. By Proposition~\ref{p:Wbk}, $\Wallman_{\!\beta\kappa}X=\Wk X$ where $\Wk X=\bigcup\{\overline{j_X(C)}:C\subset X,\;|C|\le\kappa\}$ and the closure of $j_X(C)$ is taken in $\Wallman X$.

For a $T_1$-space $X$ let $\tau^\star$ be the topology on $\Wallman_{\!\beta\kappa} X$ consisting of the sets $U\subset\Wallman_{\!\beta\kappa}  X$ such that
\begin{itemize}
\item $j_X^{-1}(U)$ is open in $X$ and
\item for every subset subset $C\subset X$ of cardinality $\le\kappa$, the set $\overline{j_X(C)}\setminus U$ is closed in $\Wk X$.
\end{itemize}
 Denote by $\Wkstar X$ the topological space $(\Wallman_{\!\beta\kappa} X,\tau^\star)$.  

The following theorem describes some properties of the space $\Wkstar X$.
 
% It is clear that the identity map $\Wkstar X\to\Wk X$ is continuous and the canonical map $j_X:X\to \Wkstar X$ remains a topological embedding.

\begin{theorem}\label{t:main-k} For any $T_1$-space $X$ the pair $(\Wkstar X,j_X)$ has the following properties:
\begin{enumerate}
\item the identity map $\Wkstar X\to\Wk X$ is continuous;
\item the map $j_X:X\to\Wkstar X$ is a topological embedding;
\item the set $j_X(X)$ is dense in $\Wkstar X$;
\item for any subspace $Z\subset \Wallman_{\!\beta\kappa} X$ of density $d(Z)\le\kappa$ the identity inclusion $Z\to \Wkstar X$ is a topological embedding;
\item the space $\Wkstar X$ is $\kappa$-bounded;
\item for any continuous map $f:X\to Y$ to a Hausdorff $\kappa$-bounded space $Y$ there exists a unique continuous map $\bar f:\Wkstar X\to Y$ such that $f=\bar f\circ j_X$;
\item the pair $(\Wkstar X,j_X)$ is a reflection of $X$ in the class of Hausdorff $\kappa$-bounded spaces if and only if the space $\Wkstar X$ is Hausdorff;
\item If $X$ is $\overline{\kappa}$-normal, then $(\Wkstar X,j_X)$ is a reflection of $X$ in the class of Hausdorff $\kappa$-bounded spaces.
\end{enumerate}
\end{theorem}

\begin{proof} 1. The definition of the topology $\tau^\star$ on $\Wallman_{\!\beta\kappa} X$ implies that this topology includes the original topology of $\Wallman_{\!\beta\kappa} X$, which means that the identity map $\Wkstar X\to \Wallman_{\!\beta\kappa} X$ is continuous.
\smallskip

2. The definition of the topology $\tau^\star$ guarantees that the map $j_X:X\to\Wkstar X$ is continuous. Taking into account that the map $j_X:X\to\Wallman_{\!\beta\kappa} X$ is a topological embedding and the identity map $\Wkstar X\to \Wallman_{\!\beta\kappa}  X$ is continuous, we conclude that the continuous map $j_X:X\to\Wkstar X$ is a topological embedding.
\smallskip

3.  To see that the space $j_X(X)$ is dense in $\Wkstar X$, take any element $u\in\Wkstar X$ and using Proposition~\ref{p:Wbk}, find a set $C\subset X$ of cardinality $|C|\le\kappa$ such that $u\in\overline{j_X(C)}$ in $\Wallman X$. By the definition of the topology $\tau^\star$, for every open neighborhood $U$ of $u$ in the space $\Wkstar X$, the set $\overline{j_X(C)}\setminus U$ is closed in $\overline{j_X(C)}\subset \Wallman X$ and hence $\overline{j_X(C)}\cap U$ is an open neighborhood of $u$ in $\overline{j_X(C)}\subset\Wallman X$. By the density of $j_X(C)$ in $\overline{j_X(C)}$ the intersection $(U\cap\overline{j_X(C)})\cap j_X(C)=U\cap j_X(C)\subset U\cap j_X(X)$ is not empty, witnessing that $j_X(X)$ is dense in $\Wkstar X$.
\smallskip

4.  Take any subspace $Z\subset \Wallman_{\!\beta\kappa} X$ of density $d(Z)\le \kappa$ and choose a dense subset $D$ in $Z$ of cardinality $|D|=d(Z)\le\kappa$. By Proposition~\ref{p:Wbk}, for every $u\in D$ there exists a set $C_u\subset X$ of cardinality $|C_u|\le\kappa$ such that $u\in\overline{j_X(C_u)}$ in $\Wallman_{\!\beta\kappa}  X$. The union $C=\bigcup_{u\in D}C_u$ has cardinality $|C|\le\kappa$ and $D\subset\overline{j_X(C)}$. Then also $Z\subset\overline{D}\subset\overline{j_X(C)}$. By the definition of the topology $\tau^\star$, for any open set $U\subset\Wkstar X$ the set $\overline{j_X(C)}\setminus U$ is closed in $\Wallman X$ and hence in $\overline{j_X(C)}$. Then $\overline{j_X(C)}\cap U$ is open in $\overline{j_X(C)}$ and $Z\cap U$ is open in $Z$. This means that the identity inclusion $Z\to\Wkstar X$ is continuous. Taking into account that the identity map $\Wkstar X\to \Wallman_{\!\beta\kappa} X$ is continuous, we conclude that the identity inclusion $Z\to \Wkstar X$ is a topological embedding.
\smallskip

5. To show that the space $\Wkstar X$ is $\kappa$-bounded, take any subset $A\subset \Wkstar X$ of cardinality $|A|\le\kappa$ and let $\bar A$ be the closure of $A$ in $\Wkstar X$. By the $\kappa$-boundedness of the space $\Wallman_{\!\beta\kappa} X$ (see Proposition~\ref{p:Wbk}), the closure $[A]$ of $A$ in $\Wallman_{\!\beta\kappa} X$ is compact. By the preceding statement, the identity inclusion $[A]\to \Wkstar X$ is a topological embedding, which implies that the set $[A]$ is compact in $\Wkstar X$. The continuity of the identity map $\Wkstar X\to \Wallman_{\!\beta\kappa}  X$ implies that $\bar A\subset[A]$. Then the space $\bar A$ is compact, being a closed subset of the compact space $[A]$.
\smallskip

6. Let $f:X\to Y$ be a continuous map of $X$ to a Hausdorff $\kappa$-bounded space $Y$. By Theorem~\ref{t:kb=kr+bk}, the Hausdorff $\kappa$-bounded space $Y$ is  $\overline{\kappa}$-regular and $\beta\kappa$-compact. %By Proposition~\ref{p:Wbk}, $\Wk X=\Wallman_{\!\beta\kappa}X$. 
By Theorem~\ref{t:ex1}, there exists a unique $\beta\kappa$-continuous map $\bar f:\Wallman_{\!\beta\kappa} X\to Y$ such that $f=\bar f\circ j_X$. We claim that the map $\bar f:\Wkstar X\to Y$ is continuous. Given any open set $U\subset Y$, we need to show that its preimage $\bar f^{-1}(U)$ belongs to the topology $\tau^\star$. 
Observe that the set $j_X^{-1}(\bar f^{-1}(U))=(\bar f\circ j_X)^{-1}(U)=f^{-1}(U)$ is open in $X$ by the continuity of the map $f$. Next, fix any subset $C\subset X$ of cardinality $|C|\le\kappa$ and consider the compact closed subset $\overline{j_X(C)}$ in $\Wallman_{\!\beta\kappa} X$. By Proposition~\ref{p:cont-d}, the restriction $\bar f{\restriction}\overline{j_X(C)}$ is continuous. Consequently, $\bar f^{-1}(U)\cap\overline{j_X(C)}=(\bar f{\restriction}\overline{j_X(C)})^{-1}(U)$ is open in $\overline{j_X(C)}$ and $\overline{j_X(C)}\setminus \bar f^{-1}(U)$ is closed in  $\overline{j_X(C)}$. Therefore, $\bar f^{-1}(U)\in\tau^\star$ and the function $\bar f:\Wkstar X\to Y$ is continuous. The uniqueness of $\bar f$ follows from the density of $j_X(X)$ in $\Wkstar X$.
\smallskip

7. The seventh statement of Theorem~\ref{t:main-k} follows immediately from the statements (2) and (6) of this theorem.
\smallskip

8. If the $T_1$-space $X$ is $\kappa$-normal, then by Proposition~\ref{p:WH} the space $\Wallman_{\!\beta\kappa} X$ is Hausdorff and so is the space $\Wkstar X$. By the preceding statement, $(\Wkstar X,j_X)$ is a Hausdorff $\kappa$-bounded reflection of $X$. 
\end{proof}

The continuous map $\Wkstar X\to\Wallman_{\!\beta\kappa} X$ is not necessarily a homeomorphism (even for $\overline{\kappa}$-normal spaces $X$). We say that a topological space $X$ is {\em regular at a point} $x\in X$ if each neighborhood of $x$ in $X$ contains a closed neighborhood of $x$ in $X$.

\begin{example} In the Cantor cube $\{0,1\}^{\w_1}$ consider the $\sigma$-products 
$$\sigma_0=\{x\in\{0,1\}^{\w_1}:|x^{-1}(1)|<\w\}\mbox{ and }\sigma_1=\{x\in\{0,1\}^{\w_1}:|x^{-1}(0)|<\w\}.$$The subspace $X=\sigma_0\cup \sigma_1$ has the following properties:
\begin{enumerate}
\item $X$ is Tychonoff and $\overline{\w}$-normal;
\item the spaces $\Wallman_{\!\beta\w} X$ and $\Wallman^\star_{\!\beta\w} X$ are Hausdorff, $\kappa$-bounded and $\overline{\kappa}$-normal;
\item the space $\Wallman_{\!\beta\w} X$ is regular at each point of the set $j_X(X)$;
\item the space $\Wallman_{\!\beta\w}^\star X$ is not regular at any point of the set $j_X(X)$;
\item the identity function $\Wallman_{\!\beta \w}X\to\Wallman_{\!\beta\w}^\star X$ is discontinuous.
\end{enumerate}
\end{example}

\begin{proof} 1. The Tychonoff space $X$ is regular and hence $\overline{\omega}$-regular. It is easy to see that the closure $\overline C$ of any countable set $C\subset X$ is countable and hence Lindel\"of. By Proposition~\ref{p:rn}, the space $X$ is $\overline{\w}$-normal.
\smallskip

2. By Proposition~\ref{p:WH}, Theorem~\ref{t:main-k}(8) and Proposition~\ref{p:cn}, the spaces $\Wallman_{\!\beta \w}X$ and  $\Wallman_{\!\beta \w}^\star X$ are Hausdorff, $\kappa$-bounded and $\overline{\kappa}$-normal. 
\smallskip

3. Fix any point $x\in X$ and take any neighborhood $W\subset\Wallman_{\!\beta\w} X$ of $j_X(x)$. We can assume that $W$ is of basic form $W=\langle V\rangle=\{u\in\Wallman_{\!\beta\w} X:\exists F\in u\;(F\subset V)\}$ for some open neighborhood $V$ of $x$ in $X$. By the regularity of the space $X$, there exists an open neighborhood $U$ of $x$ in $X$ such that $\overline{U}\subset V$. Then $\langle U\rangle$ is a neighborhood of $j_X(x)$ in $\Wallman_{\!\beta\w} X$ such that $\overline{\langle U\rangle}\subset \langle V\rangle=W$.
\smallskip

4. Given any point $x\in \sigma_1$, we shall prove that the space $\Wallman^\star_{\!\beta\w}X$ is not regular at $j_X(x)$. Since $x\in\sigma_1$, there is an infinite ordinal $\alpha_x\in\w_1$ such that $x(\alpha)=1$ for any $\alpha\in[\alpha_x,\w_1)$. For every ordinal $\alpha\in [\alpha_x,\w_1]$ consider the function $z_\alpha:\w_1\to\{0,1\}$ such that for every $\gamma\in\w_1$ 
$$z_\alpha(\gamma)=\begin{cases}x(\gamma)&\mbox{if $\gamma\in[0,\alpha_x)$};\\
1,&\mbox{if $\gamma\in[\alpha_x,\alpha)$}\\
0,&\mbox{if $\gamma\in[\alpha,\w_1)$}.
\end{cases}$$
Consider the subspaces $L=\{z_\alpha:\alpha\in[\alpha_x,\w_1)\}$ and $\bar L=L\cup\{x\}=L\cup\{z_{\w_1}\}$ of $\{0,1\}^{\w_1}$. It is easy to see that $\bar L$  is homeomorphic to the compact Hausdorff space $[\alpha_x,\w_1]$ endowed with the interval topology. Moreover, $L\subset\Sigma_0\setminus X$ where 
$$\Sigma_0=\{x\in\{0,1\}^{\w_1}:|x^{-1}(1)|\le\w\}\mbox{ and }\Sigma_1=\{x\in\{0,1\}^{\w_1}:|x^{-1}(0)|\le \w\}.$$
Since the subspace $\Sigma=\Sigma_0\cup\Sigma_1$ of $\{0,1\}^{\w_1}$ is regular and $\w$-bounded, by Corollary~\ref{c:rex}, there exists a continuous map $f:\Wallman_{\!\beta\w} X\to\Sigma$ such that $f\circ j_X=\mathrm{id}_X$. 

We claim that the set $U=(\Wallman_{\!\beta\w} X)\setminus f^{-1}(L)$ is open in $\Wallman_{\!\beta\w}^\star X$.
First observe that $j_X^{-1}(U)=X$ is open in $X$. Next, take any countable set $C\subset X$. Observe that for every $i\in\{0,1\}$ the closure of the set $C\cap\sigma_i$ in $\Sigma_i$ is compact and metrizable. Then the closure $\overline{C}=\overline{C\cap\sigma_0}\cup\overline{C\cap\sigma_1}$ is compact and metrizable, too. Since the space $L$ is countably compact, the intersection $L\cap\bar C$ is compact, being a countably compact metrizable space. The equality $f\circ j_X=\mathrm{id}_X$ implies $f(\overline{j_X(C)})\subset \overline{f\circ j_X(C)}=\overline{C}$. Then 
$$\overline{j_X(C)}\setminus U=\overline{j_X(C)}\cap f^{-1}(L)=\overline{j_X(C)}\cap f^{-1}(L\cap\overline{C})$$ is  closed in $\overline{j_X(C)}$, witnessing that the set 
$U$ is open in $\Wallman_{\!\beta\w}^\star X$. Assuming that the space $\Wallman_{\!\beta\w}^\star X$ is regular at $j_X(x)$, we can find an open neighborhood $W$ of $j_X(x)$ in $\Wallman_{\!\beta\w}^\star X$ whose closure $\overline{W}$ in $\Wallman_{\!\beta\w}^\star X$ is contained in the open set $U$. Since $j_X:X\to \Wallman_{\!\beta\w}^\star X$ is a topological embedding, the set $V=j_X^{-1}(W)$ is a neighborhood of $x$ in $X$. By the denisty of the space $X$ in the regular space $\Sigma$, the closure $\overline{V}$ of $V$ in $\Sigma$ is a neighborhood of $x$ in $\Sigma$. Since $x\in\bar L$, we can find an ordinal $\alpha\in [\alpha_x,\w_1)$ such that $z_\alpha\in \overline{V}$. Since the set $\sigma_0$ is dense in $X$, we get $z_\alpha\in \overline{V}=\overline{V\cap\sigma_0}$. By \cite[3.10.D]{Eng}, the space $\Sigma_0$ is Fr\'echet-Urysohn. Consequently, there exists a sequence $(x_n)_{n\in\w}$ in $V\cap\sigma_0$ that converges to $z_\alpha$. Fix any ultrafilter $p\in\beta\w\setminus\w$. By the $\w$-boundedness of the space $\Wallman_{\beta\w}^\star X$, the sequence $(j_X(x_n))_{n\in\w}$ has a $p$-limit $\bar x$ in $\Wallman_{\beta\w}^\star X$. The continuity of the map $f$ ensures that $f(\bar x)\in \lim_p(f\circ j_X(x_n))_{n\in\w}=\lim_p(x_n)_{n\in\w}=\{z_\alpha\}\subset L$ and hence $\bar x\notin U$. On the other hand, $\bar x$ belongs to $\overline{W}\subset U$ being a $p$-limit point of the sequence $\{j_X(x_n)\}_{n\in\w}\subset W$. This contradiction shows that $\Wallman_{\beta\w}^\star X$ is not regular at $x$. By analogy we can prove that  $\Wallman_{\beta\w}^\star X$ is  not regular at any point of the set $j_X(\sigma_0)$.
\smallskip

5. Assuming that the identity function $\Wallman_{\!\beta \w}X\to\Wallman_{\!\beta\w}^\star X$ is continuous, we conclude that it is a homeomorphism. Then the regularity of the space $\Wallman_{\!\beta \w}X$ at points of the set of $j_X(X)$ implies the regularity of the space $\Wallman_{\!\beta \w}^\star X$ at points of $j_X(X)$. But this contradicts the statements (3) and (4).
\end{proof}

\section{The Hausdorff and $\overline{\kappa}$-Urysohn $M$-compact reflections of a topological space}\label{s:main-M}

Let $\kappa$ be an infinite cardinal and $M$ be a nonempty subset of $\beta\kappa$.  By \cite[3.2]{GFK}, the class $\HM$ of $M$-compact Hausdorff spaces is productive, closed-hereditary and topological. This implies that the class  $\UM$ of $\overline{\kappa}$-Urysohn $M$-compact spaces is also productive, closed-hereditary and topological. 

By Theorem~\ref{t:Kak}, each topological space $X$ has a unique $\HM$-reflection $(\HM X,i_X)$ and a unique $\UM$-reflection $(\UM X,i_X)$. In this section we show that for a $\overline{\kappa}$-normal $T_1$-space $X$ (containing no long $\kappa$-butterflies), the $\UM$-reflection (and $\HM$-reflection) of $X$ can be identified with the Wallman $M$-compact extension $\WM X$ endowed with the topology $\tau_M^\sharp$ defined as follows.

For a $T_1$-space $X$ let $\tau^\sharp_M$ be the topology on $\WM X=\cl_{[M]}^1 j_X(X)\subset \Wallman X$ consisting of the sets $U\subset \WM X$ such that
\begin{itemize}
\item $j_X^{-1}(U)$ is open in $X$ and
\item $\WM X\setminus U$ is $[M]$-closed in $\WM X$.
\end{itemize}
 Denote by $\WMstar X$ the topological space $(\WM X,\tau^\sharp_M)$.  
The following theorem describes some properties of this topological space.
 
% It is clear that the identity map $\Wkstar X\to\Wk X$ is continuous and the canonical map $j_X:X\to \Wkstar X$ remains a topological embedding.

\begin{theorem}\label{t:main-M} For any $T_1$-space $X$ the pair $(\WMstar X,j_X)$ has the following properties.
\begin{enumerate}
\item The identity map $\WMstar X\to\WM X$ is continuous.
\item The identity map $\WM X\to \WMstar X$ is $[M]$-continuous.
\item The map $j_X:X\to\WMstar X$ is a topological embedding.
\item The set $j_X(X)$ is dense in $\WMstar X$.
%\item for any subspace $Z\subset \WM X$ of tightness $t(Z)\le\kappa$ the idenity inclusion $Z\to \WMstar X$ is a topological embedding;
\item The space $\WMstar X$ is $M$-compact.
\item If the space $X$ is $\overline{\kappa}$-normal, then the space $\WMstar X$ is Hausdorff and $\overline{\kappa}$-Urysohn.
\item If the space $\WM X$ is Hausdorff, then any $M$-compact subspace of $\WM X\setminus j_X(X)$ is closed $\WMstar X$.
\item For any continuous map $f:X\to Y$ to a $\overline{\kappa}$-Urysohn $M$-compact $T_1$-space $Y$ there exists a unique continuous map $\bar f:\WMstar X\to Y$ such that $f=\bar f\circ j_X$.
\item If $X$ contains no long $\kappa$-butterflies, then for any continuous map $f:X\to Y$ to a Hausdorff $M$-compact space $Y$ there exists a unique continuous map $\bar f:\WMstar X\to Y$ such that $f=\bar f\circ j_X$.
\item If $X$ is $\overline{\kappa}$-normal, then $(\WMstar X,j_X)$ is a reflection of $X$ in the class $\UM$ of $\overline{\kappa}$-Urysohn $M$-compact spaces.
\item If $X$ is $\overline{\kappa}$-normal and contains no long $\kappa$-butterflies, then $(\WMstar X,j_X)$ is a reflection of $X$ in the class $\HM$ of Hausdorff $M$-compact spaces.
\end{enumerate}
\end{theorem}

\begin{proof} 1. The definition of the topology $\tau^\sharp_M$ on $\WM X$ implies that this topology includes the original topology of $\WM X$, which means that the identity map $\WMstar X\to \WM X$ is continuous.
\smallskip

2. To see that the identity map $\WM X\to \WMstar X$ is $[M]$-continuous, take any ultrafilter $p\in[M]$ and a $\kappa$-sequence $x:\kappa\to \WM X$ that $p$-converges to a point $u\in\WM X$ in the space $\WM x$. We need to prove that the $\kappa$-sequence $x$ is $p$-convergent to $u$ in the space $\WMstar X$. Assuming that this is not true, we can find an open neighborhood $U$ of $u$ in the space $\WMstar X$ such that the set $\{\alpha\in\kappa:x(\alpha)\in U\}$ does not belong to the ultrafilter $p$. Then the set $P=\{\alpha\in\kappa:x(\alpha)\notin U\}$ belongs to $p$. Take any $\kappa$-sequence $y:\kappa\to \WM X\setminus U$ such that $y{\restriction}P=x{\restriction}P$. It can be shown that the $\kappa$-sequence $y$ is $p$-convergent to $u$. By the $[M]$-closedness of the set $\WM X\setminus U$, we have $u\in \WM X\setminus U$. But this contradicts the choice of the neighborhood $U$ of $u$.
\smallskip

3. The definition of the topology $\tau^\sharp_M$ guarantees that the map $j_X:X\to\WMstar X$ is continuous. Taking into account that the map $j_X:X\to\WM X$ is a topological embedding and the identity map $\WMstar X\to \WM X$ is continuous, we conclude that the continuous map $j_X:X\to\WMstar X$ is a topological embedding.
\smallskip

4.  To derive a contradiction, assume that set $j_X(X)$ is not dense in $\WMstar X$ and find a nonempty open set $U\subset \WMstar X$ such that $U\cap j_X(X)=\emptyset$. Fix any ultrafilter $u\in U$. Since $u\in \WM X=\cl_{[M]}^1j_X(X)$, there exists an ultrafilter $p\in[M]$ and a $\kappa$-sequence $x:\kappa\to j_X(X)$ such that $u\in\lim_p x$.  Observe that $x(\kappa)\subset j_X(X)\subset \WM X\setminus U$. By the definition of the topology $\tau^\sharp_M$, the set $\WM X\setminus U$ is $[M]$-closed in $\WM X$. Then $u\in \lim_p x\subset \WM X\setminus U$, which contradicts the choice of $u$. This contradiction witnesses that the set $j_X(X)$ in dense in $\WMstar X$.
\smallskip

5. By Proposition~\ref{p:clM}, the space $\WM X$ is $M$-compact. By Lemma~\ref{l:MMim}, the space $\WMstar X$ is $M$-compact being the image of the $M$-compact space $\WM X$ under the $M$-continuous map $\WM X\to\WMstar X$.
\smallskip

6. If the space $X$ is $\overline{\kappa}$-normal, then by Proposition~\ref{p:WH},  the space $\WM X$ is Hausdorff and $\overline{\kappa}$-Urysohn. Since the identity map $\WMstar X\to \WM X$ is continuous, the space $\WMstar X$ is $\overline{\kappa}$-Urysohn, too.
\smallskip

7. Assume that the space $\WM X$ is Hausdorff and $K\subset \WM X\setminus j_X(X)$ is an $M$-compact subspace of $\WM X$. By Lemma~\ref{l:[M]}, the $M$-compact space $K$ is $[M]$-compact and by Lemma~\ref{l:Mk=>Mc}, the set $K$ is $[M]$-closed in $\WM X$. Since $K\cap j_X(X)=\emptyset$, the set $K$ is closed in $\WMstar X$ by the definition of the topology $\tau^\sharp_M$.

\smallskip
8. Let $f:X\to Y$ be a continuous map of $X$ to a  $\overline{\kappa}$-Urysohn $M$-compact $T_1$-space $Y$. By Theorem~\ref{t:ex1}, there exists a unique $[M]$-continuous map $\bar f:\WM X\to Y$ such that $f=\bar f\circ j_X$. We claim that the map $\bar f:\WMstar X\to Y$ is continuous. Given any open set $U\subset Y$, we need to show that its preimage $\bar f^{-1}(U)$ belongs to the topology $\tau^\sharp_M$.
This will follow as soon as we check two conditions:
\begin{itemize}
\item $j^{-1}_X(\bar f^{-1}(U))$ is open in $X$ and
\item $\WM X\setminus \bar f^{-1}(U)$ is $[M]$-closed in $\WM X$.
\end{itemize}
The set $j_X^{-1}(\bar f^{-1}(U))=(\bar f\circ j_X)^{-1}(U)=f^{-1}(U)$ is open in $X$ by the continuity of the map $f$. To show that the set $\WM X\setminus\bar f^{-1}(U)$ is $[M]$-closed in $\WM X$, fix any ultrafilter $p\in[M]$ and $\kappa$-sequence $x:\kappa\to \WM X\setminus\bar f^{-1}(U)$. We need to prove that $\lim_p x\subset \WM X\setminus \bar f^{-1}(U)$. Take any point $\bar x\in\lim_p x$. The $[M]$-continuity of the map $\bar f$ and the closedness of the set $Y\setminus U\supset \bar f\circ x(\kappa)$ ensure that $\bar f(\bar x)\in\lim_p(\bar f\circ x)\subset Y\setminus U$ and hence $\bar x\in \bar f^{-1})(Y\setminus U)=\WM X\setminus \bar f^{-1}(U)$. Therefore, the set $\bar f^{-1}(U)$ belongs to the topology $\tau^\sharp_M$ and hence the function $\bar f:\WMstar X\to Y$ is continuous.

The uniqueness of $\bar f$ follows from the density of $j_X(X)$ in $\WMstar X$.
\smallskip

9. Assume that $X$ contains no $\kappa$-butterflies and let $f:X\to Y$ be any continuous map to a Hausdorff $M$-compact space $Y$. By Theorem~\ref{t:ex1}, there exists a unique $[M]$-continuous map $\bar f:\WM X\to Y$ such that $f=\bar f\circ j_X$. Repeating the argument from the proof of the preceding statement, we can verify that the map $\bar f:\WMstar X\to Y$ is continuous. The uniqueness of $\bar f$ follows from the density of  $j_X(X)$ in $\WMstar X$.
\smallskip

10. If the $T_1$-space $X$ is $\kappa$-normal, then by Proposition~\ref{p:WH} the space $\Wk X$ is $\kappa$-Urysohn and so is the space $\WMstar X$. Now the  statement (8) implies that $(\WMstar X,j_X)$ is a reflection of $X$ in the class $\UM$ of $\overline{\kappa}$-Urysohn $M$-compact spaces. 
\smallskip

11. If the $T_1$-space $X$ is $\kappa$-normal and contains not $\kappa$-butterflies then by Proposition~\ref{p:WH} the space $\Wk X$ is Hausdorff and so is the space $\WMstar X$. Now the statement (9) implies that $(\WMstar X,j_X)$ is a reflection of $X$ in the class $\HM$ of Hausdorff $M$-compact spaces. 
\end{proof}

Observe that for any separable normal $T_1$-space $X$ its Hausdorff $\w$-bounded reflexion $\mathsf H_\w X$ coincides with the Wallman compactification $\Wallman X=\beta X$ of $X$.  So, the structure of  $\mathsf H_\w X$ is well-understood. On the other hand, even for a countable discrete space $X$ the structure of its $\mathsf H_{\!\beta\w}$- and $\mathsf U_{\!\beta\w}$-reflections $\mathsf H_{\!\beta\w}X=\mathsf U_{\!\beta\w} X=\Wallman_{\!\beta\w}^\sharp X$ is rather mysterious.

\begin{proposition} If an $\overline{\w}$-normal $T_1$-space $X$ is not countably compact, then its $\mathsf H_{\beta\w}$-reflection $\mathsf H_{\beta\w} X$ and $\mathsf U_{\!\beta\w}$-reflection $\mathsf U_{\!\beta\w} X$ are not $\overline{\w}$-regular and not $\omega$-bounded. Consequently, neither $\mathsf H_{\beta\w} X$ nor $\mathsf U_{\!\beta\w} X$ is homeomorphic to $\Wallman_{\!\beta\w} X$ or $\Wallman_{\!\beta\w}^\star X=\mathsf H_\kappa X$.
\end{proposition}

\begin{proof} By Theorem~\ref{t:main-M}(5,6), the Wallman $\beta\w$-compact extension $\Wallman_{\!\beta\w}^\sharp X$ of the $\overline{\w}$-normal $T_1$-space $X$ is $\beta\w$-compact and $\overline{\w}$-Urysohn. By the definition of the $\mathsf H_{\beta\w}$-reflection $(\mathsf H_{\beta\w}X,i_X)$ of $X$, there exists a continuous map $f:\mathsf H_{\beta\w}X\to \Wallman_{\!\beta\w}^\sharp X$ such that $f\circ i_X=j_X$.

The normal $T_1$-space $X$ is not countably compact and hence it admits a closed topological embedding $e:\w\to X$ of the countable discrete space $\w$. The map $e:\w\to X$ has a continuous extension $\bar e:\Wallman \w\to\Wallman_{\!\beta\w} X$, which is a closed topological embedding of $\Wallman \w=\beta\w$ into the Wallman $\beta\w$-extension $\Wallman_{\!\beta\w} X$ of $X$.
By a famous result of Kunen \cite{Kunen} (see also \cite[4.3.4]{vM}), the space $\w^*=\Wallman \w\setminus \w$ contains a point $p$ such that the complement $C=\w^*\setminus\{p\}$ is $\w$-bounded (such points $p$ are called weak $P$-points). Then the space $C=\bar e(\w^*\setminus\{p\})$ is $\w$-bounded and by  Theorem~\ref{t:main-M}(6), the set $C$ is closed in $\Wallman_{\!\beta\w}^\sharp X$. Then its preimage $f^{-1}(C)$ is a closed subset of $\mathsf H_{\beta\w}X$ and so is the intersection $F=f^{-1}(C)\cap \overline{i_X(e(\w))}$.

By the $\beta\w$-compactness of the space $\mathsf H_{\beta\w} X$ the $\w$-sequence $i_X\circ e:\w\to i_X(X)\subset \mathsf H_{\beta\w} X$ has a unique $p$-limit point $\bar p\in\lim_p(i_X\circ e)\subset\overline{i_X\circ e(\w)}$. The continuity of the map $f$ ensures that $f(\bar p)\in\lim_p (f\circ i_X\circ  e)=\lim_p(j_X\circ e)=\{\bar e(p)\}\notin C$ and hence $\bar p\notin F$. Assuming that the space $\mathsf H_{\beta\w} X$ is $\overline{\omega}$-regular, we can find an open neighborhood $V$ of $\bar p$ in $\mathsf H_{\beta\w} X$ such that $\overline{V}\cap F=\emptyset$. Since $\bar p\in V\cap\lim_p(i_X\circ e)$, the set $P=\{n\in\w:i_X\circ e(n)\in V\}$ belongs to the ultrafilter $p$ and hence is infinite. Chose any  free ultrafilter $q\in \beta\w$ such that $P\in q$ but $q\ne p$. By the $\beta\w$-compactness of $\mathsf H_{\beta\w} X$, the $\w$-sequence $i_X\circ e$ has a unique $q$-limit point  $\bar q$. It follows that $\bar q\in\lim_q (i_X\circ e)\subset\overline{i_X\circ e(P)}\subset\overline{V}$ and hence $\bar q\notin F$, $f(\bar q)\notin C$, and
$f(\bar q)\in \bar e(\w^*)\setminus C=\{\bar e(p)\}$. On the other hand, the continuity of the map $f$ guarantees that $f(\bar q)\in \lim_q(f\circ i_X\circ e)=\lim_q(j_X\circ e)=\bar e(q)\ne \bar e(p)$, by the injectivity of $\bar e$.
Therefore, the space $\mathsf H_{\beta\w} X$ is not $\overline{\w}$-regular. By Proposition~\ref{p:cn}, the space  $\mathsf H_{\beta\w} X$ is not $\w$-bounded, and it is not homeomorphic to the $\w$-bounded spaces $\Wallman_{\!\beta\w} X$ or $\Wallman_{\!\beta\w}X=\mathsf H_\w X$.

By analogy we can prove that the reflection $(\mathsf U_{\!\beta\w} X,i_X)$ of $X$ in the class $\mathsf U_{\!\beta\w}$ of $\overline{\kappa}$-Urysohn $\beta\w$-compact spaces is not $\overline{\w}$-regular, not $\w$-bounded, and not  homeomorphic to the $\w$-bounded spaces $\Wallman_{\!\beta\w} X$ or $\Wallman_{\!\beta\w}X=\mathsf H_\w X$.
\end{proof}

\begin{problem} Let $X$ be a (metrizable) $\overline{\w}$-normal $T_1$-space. Is the pair $(\Wallman^\sharp_{\!\beta\w} X,j_X)$ a reflection of $X$ in the class of Hausdorff $\beta\w$-compact spaces?
\end{problem}

\section*{Acknowledgements}
The  authors would like to express their sincere thanks to Serhii Bardyla for motivating discussions leading to writing this paper.

\end{document}